\documentclass[12pt]{amsart}
\usepackage{amsmath}           
\usepackage{amssymb}

\usepackage{xcolor}

\DeclareMathOperator{\im}{{\sf im}}
\DeclareMathOperator{\End}{{\sf End}}

\DeclareMathOperator{\lt}{\sf L}

\DeclareMathOperator{\id}{{\sf id}}

\DeclareMathOperator{\Q}{{\sf Q}}

\begin{document}

\newcommand{\Sp}{{\rm Span}} 
\newtheorem{thm}{Theorem}   
\newtheorem{lemma}[thm]{Lemma}   
\newtheorem{proposition}[thm]{Proposition}   
\newtheorem{theorem}[thm]{Theorem}   
\newtheorem{corollary}[thm]{Corollary}   
\newtheorem{pro}[thm]{Proposition}
\newtheorem{cor}[thm]{Corollary}
\newtheorem{lem}[thm]{Lemma}
\newtheorem{claim}[thm]{Claim}
\newtheorem{fact}[thm]{Fact}
\newtheorem{definition}[thm]{Definition}
\newtheorem{observation}[thm]{Observation}
\newcommand{\red}{\textcolor[rgb]{1.00,0.00,0.00}}

\newcommand{\lab}[1]{\label{#1}}
\renewcommand{\labelenumi}{{\textrm{(\roman{enumi})}}}
\newcommand{\vep}{\varepsilon} 
\newcommand{\One}{1} 
\newcommand{\Zero}{0} 
  \newcommand{\mb}[1]{\mathbb{#1}}
  \newcommand{\mc}[1]{\mathcal{#1}}
  \renewcommand{\phi}{\varphi}
\newcommand{\Qm}{\Q \mc{M}_f}
\newcommand{\et}{\wedge}
\newcommand{\imp}{\rightarrow  }
 \newcommand{\ex}{\exists  }
 \newcommand{\all}{\forall  }
\newcommand{\Mod}{{\sf Mod } }
\newcommand{\al  }{\alpha }
\newcommand{\be  }{\beta }
\newcommand{\ga  }{\gamma }
\newcommand{\de  }{\delta }
\newcommand{\Si}{\Sigma  }
\newcommand{\no}{\noindent}
\newcommand{\sub}{\subseteq}
\renewcommand{\th}{\theta}
\newcommand{\eqq}{\mb{E}}
\newcommand{\xc}[1]{}
\newcommand{\la}{{\sf L}}
\newcommand{\vb}{{{\sf L}(H)}}

\newcommand{\rh}{\mc{R}}
\newcommand{\lh}{\mc{L}}

\newcommand{\fs}{{${\rm FEAS}_{\mathbb{Z},\mathbb{R}}$}}
\newcommand{\hr}{\mc{H}_{\mathbb{R}}}
\newcommand{\vv}{ {\sf L}(V_F)} 
\newcommand{\bp}{{\rm BP}$\mc{NP}^0_\mathbb{R}$}

 \newcommand{\La}{{\sf L}}\newcommand{\mr}{\mc{R}}

\newcommand{\qr}{$\ast$-regular ring }
\newcommand{\ra}{\rightarrow}


\title[Equational decision problems]{On the complexity
of  equational decision problems for finite height
(ortho)complemented modular lattices}

\author[C. Herrmann]{Christian Herrmann}
\address{Technische Universit\"{a}t Darmstadt FB4\\Schlo{\ss}gartenstr. 7, 64289 Darmstadt, Germany}
\email{herrmann@mathematik.tu-darmstadt.de}

\begin{abstract}
We study the computational complexity of  the 
satisfiability problem and the complement of the equivalence problem for
 complemented (orthocomplemented)
modular lattices $L$ and classes thereof.
Concerning  a  simple  $L$ of finite height,
 $\mc{NP}$-hardness 
is shown for both  problems.
Moreover, both problems are   shown to be 
polynomial-time equivalent to the same  feasibility problem over the
division ring $D$ whenever
$L$ is  the  subspace lattice of a $D$-vector
space  of finite dimension at least $3$.

Considering the class of all finite dimensional Hilbert spaces,
the equivalence problem for the class of subspace ortholattices
is shown to be polynomial-time equivalent 
to that for   the class of endomorphism $*$-rings 
with pseudo-inversion;
moreover, we  derive   completeness for the
complement of the
Boolean part of the  nondeterministic Blum-Shub-Smale model of real 
computation 
without constants. 
This result extends to the
additive category of finite dimensional Hilbert spaces,
enriched by adjunction and pseudo-inversion.
\end{abstract}

\dedicatory{
Dedicated to the honor of Ralph Freese, Bill Lampe, and J.B.Nation}


\maketitle

\section{Introduction}
Given a class $\mc{A}$ of algebraic structures,
the \emph{equational theory}
${\sf Eq}(\mc{A})$ of $\mc{A}$
consists of all identities valid in all members of
$\mc{A}$, and so in the variety ${\sf V}(\mc{A})$ 
generated by $\mc{A}$. The associated \emph{decision problem}
asks, for any given identity, whether or not it is in ${\sf Eq}(\mc{A})$.
This problem is also known as the \emph{equivalence problem} for
$\mc{A}$.
The  \emph{triviality problem}
for $\mc{A}$ is to decide 
for each given  finite presentation
whether the associated freely generated 
 member of  ${\sf V}(\mc{A})$
is trivial  or not. For a survey of decision problems in
Universal Algebra see \cite{stan0}.

Generalizing the well known  Boolean case, in \cite{jacm}
the following decision problems 
have been considered. \emph{Refutability}
(corresponding to ``weak satisfiability'' in \cite{jacm}) {\rm REF}$_\mc{A}$: 
Given  terms $t(\bar x),s(\bar x)$
is there $A \in \mc{A}$ and $\bar a$ in $A$
such that $A\models t(\bar a) \neq s(\bar a)$.
 \emph{Satisfiability} (corresponding to ``strong satisfiability'' in \cite{jacm})  SAT$_\mc{A}$:
Given  terms $t_i(\bar x),s_i(\bar x)$, $i=1,\ldots, n$,
is there $A \in \mc{A}$ and  $\bar a$ in $A$
such that
 $A\models  t_i(\bar a)=s_i(\bar a)$ for all $i$
and such that  the entries of $\bar a$ generate a non-trivial subalgebra
(this is to exclude trivial assignments when dealing with  lattices).
These decision problems are   polynomial-time
(shortly ``p-time'') equivalent to the  complement
of the decision problem for the equational  theory and
the triviality problem for $\mc{A}$, respectively.
We write SAT$_A$ and REF$_A$ if $\mc{A}=\{A\}$.

 According to Proposition 1.16 and
 the proof of Theorem 2.11 in
\cite{jacm},
SAT$_L$ and REF$_L$ 
 are p-time equivalent and $\mc{NP}$-hard 
  whenever $L$ is a simple
modular ortholattice of finite height.  
We are going to derive in Section~\ref{nph} $\mc{NP}$-hardness
of SAT$_L$ and REF$_L$ for
simple complemented modular lattices $L$ (bounds $0,1$
in the signature or not)
 of finite height.
The key is the case of the $2$-element lattice $\textbf{2}$.
$\mc{NP}$-completeness of REF$_L$  for  $L=\textbf{2}$
and even any finite lattice $L$
is   Theorem 3.3 in Bloniarz, Hunt, and Rosenkrantz \cite{blon0}.
For $L=\textbf{2}$  we rely
on a simple proof due to Ross Willard.

For the subspace lattice $\lt(V_F)$ of an $F$-vector space,
 we show in Section~\ref{fsl} that
both SAT$_{\lt(V_F)}$ and
REF$_{\lt(V_F)}$ are  p-time equivalent to FEAS$_{\mathbb{Z},F}$
provided that $3 \leq \dim V_F < \infty$.

Here, for a ring $R$, the problem FEAS$_{\mathbb{Z},R}$
is
to decide, for any finite list 
of terms $p_i(\bar x)$,
 in the signature of rings with $0,1$, $+,-,\cdot$, whether there is a common zero within $R$. 
In case of commutative $R$, the $p_i(\bar x)$
can be  replaced by
 multivariate polynomials $p_i$ in commuting variables and
 with integer coefficients.

Focus of \cite{jacm} was on the class $\mc{H}$
of all finite
dimensional real and complex  Hilbert spaces 
(that is, Euclidean and unitary spaces) 
and the class $\lh$
of subspace ortholattices 
${\sf L}(H)$, $H \in \mc{H}$;
that is, ${\sf L}(H)$
is the lattice of all linear subspaces of $H$,
with constants $0$ and  $H$, 
and orthocomplementation $U \mapsto U^\perp$,
the orthogonal of $U$.
Here, for any fixed  $H\in \mc{H}$,
${\rm REF}_{{\sf L}(H)}$ and 
${\rm SAT}_{{\sf L}(H)}$ are both  decidable due to
Tarski's decision procedure for $\mathbb{R}$;
for $\dim H\geq 3$,
 the complexity  has been determined, in \cite[Theorem 2.7]{jacm},
within the Blum-Shub-Smale model of non-deterministic
computation over the reals: both problems are
complete for 
  BP$\mc{NP}^0_\mathbb{R}$, 
the part of the model  which allows only integer 
 constants and  binary instances. 
This class  contains $\mc{NP}$ and 
  is, within this model, polynomial-time equivalent
to the problem FEAS$_{\mathbb{Z},\mathbb{R}}$, see \cite{jacm} for references.

The decision problem for  ${\sf Eq}(\lh)$ was  shown solvable in \cite{hard2,neu}, REF$_{\lh} \in  {\rm BP}\mc{NP}^0_\mathbb{R}$ 
in  \cite[Theorem 4.4]{jacm}. 
On the other hand, SAT$_{\lh}$ was shown   undecidable in \cite{SAT},
as well as
SAT$_\mc{C}$  for any class $\mc{C}$ 
of (expansions of)  modular lattices containing some 
subspace lattice ${\sf L}(V_F)$ of a vector space $V_F$
where $\dim V_F$ is infinite
or contains all  ${\sf L}(V_F)$ where $F$ is of characteristic $0$
and $\dim(V_F)$ finite.

The main objective of the present note 
was to  answer  \cite[Question 4.5]{jacm}, 
namely to  show (in Section~\ref{hilb}) that
REF$_{\lh}$ is p-time equivalent to FEAS$_{\mathbb{Z},\mathbb{R}}$. 
In Section~\ref{ring}, FEAS$_{\mathbb{Z},\mathbb{R}}$
 is also shown p-time equivalent
to REF$_\mc{R}$, where $\mc{R}$ is
the class  of all    $*$-rings
(with pseudo-inversion) $\End(H)$ of endomorphisms 
of $H$, $H \in \mc{H}$. This result extends to the
additive category of finite dimensional Hilbert spaces,
enriched by adjunction and pseudo-inversion.

As general references for lattice theory  we refer to
\cite{berb,berb2,cd,mae,sko}, for decision problems
 to \cite{stan}
Non-determinism in real computation will not enter into our discussion, explicitly.
That is, decision problems are understood in the sense of Logic
and related by p-time  reductions,
given by translations of formulas to be carried out
by a Turing machine in polynomial-time.
 In particular, this
applies when considering {\fs}.

Thanks are due to   referees and editor for
detailed corrections and  valuable suggestions,
in particular the hint to \cite{blon0}.

Best wishes go to Bill, J.B., and Ralph; also sincere thanks 
for hospitality, inspiring discussions, and 
excellent refereeing.

\section{Preliminaries}

\subsection{Basic equations and unnested terms}\lab{u2}
Translation from one equational language into another means, in essence,
to translate basic equations. Also, reducing a problem
first to one formulated via basic equations,
avoids possible exponential blowup in translations.
This technique is well known, compare \cite{stan}.

A \emph{basic equation} is of the form $y=x$ or
 $y=f(\bar x)$
where $f$ is a fundamental operation symbol and $\bar x$ 
a string of variables, the length of which is the arity of $f$.
A \emph{simple formula} $\phi(\bar x)$ in (a string of free) variables
$\bar x$ is of the form $\exists \bar u\,\phi_0(\bar x, \bar u)$
where $\phi_0(\bar x,\bar u)$ is a conjunction of
basic equations in variables $\bar x,\bar u$.

Given a  first order language,
the set $\Omega(\bar x)$ of \emph{circuits} in
\emph{input} variables $\bar x$
and \emph{output} variables $\bar y$ (here, $\bar x$ and $\bar y$
are lists without repetition)
consists of pairs $(\phi,\bar y)$, where $\phi$ is a formula, according to the following inductive definition:
\begin{itemize}
\item $(\emptyset,\bar x) \in\Omega(\bar x)$  for the empty conjunction $\emptyset$. 
\item Assume that 
$(\phi,\bar y) \in    \Omega(\bar x)$
and that $\psi$ is an equation 
$y=f(y_1,\ldots,y_m)$  where
 $f$ is an $m$-ary fundamental operation,
 each $y_i$ occurs in $\bar y$, and
$y$ is a new variable; 
then $(\phi',\bar y') \in \Omega(\bar x)$
where $\phi'$ is the conjunction of $\phi$ and $\psi$
and where $\bar y'$ is obtained from $\bar y$ 
by possibly omitting some of the $y_i$ occurring in $\psi$ and
by adding $y$. 
\end{itemize}
This is just a variant of the concept of algebraic circuit.
Let the set $\Theta(\bar x)$
of \emph{unnested terms} 
consist of the $(\emptyset,x_i)$ and
the $(\phi,y)\in \Omega(\bar x)$, that is, those with 
singleton  output variable $y$. 
For $T=(\phi_T,y_T) \in \Theta(\bar x)$,
that is with input variables $\bar x$ and output variable
$y_T$, 
let ${\bar u}_T$ denote the variables occurring in $\phi_T$
which are distinct from $y_T$ and the $x_i$.
Then $\phi_T$ is of the form $\phi(\bar x,y_T,{\bar u}_T)$ and for any 
algebraic structure $A$ (of the relevant similarity type)
 and assignment $\bar x \mapsto \bar a$ in $A$
there is  a unique $b \in A$ such that $A\models \phi_T(\bar a,b,\bar c)$
for some (also unique) $\bar c$ in $A$. Accordingly, we 
write $T=T(\bar x)$ and 
$T(\bar a)=b$ and use
$T(\bar x)=y$ as an abbreviation for the simple formula
$\exists  y_T\exists {\bar u}_T.\,\phi_T(\bar x,y_T,{\bar u }_T)
\wedge y_T=y$. 
The length of  terms $t(\bar x)$  and \emph{unnested terms}
$T(\bar x)$ 
will be denoted by $|t(\bar x) |$ and $|T(\bar x)|$. 
Let $o(t(\bar x))$ and $o(T(\bar x))$ the number
of occurrences of variables from $\bar x$ in $t(\bar x)$ and $T(\bar x)$,
respectively. Observe that all variables listed in $\bar x$
are  considered to  occur  in $t(\bar x)$.

\begin{fact}\lab{ured}
There is a map $\theta$, associating in p-time
with a  term $t(\bar x)$  an unnested term
$T(\bar x)$, of length linear
in that of $t(\bar x)$, 
such that $t(\bar x)=y$ is equivalent to $T(\bar x)=y$
and such that $o(T(\bar x))=o(t(\bar x))$.
For any unnested term $T(\bar x)$ 
there is a term $t(x)$ such that $\theta(t(\bar x))=T(\bar x)$. 
\end{fact}
Thus,
REF$_{\mc{A}}$  reduces in p-time to its
analogue
 uREF$_{\mc{A}}$
for unnested terms: to decide for any $T(\bar  x)$ and
$S(\bar x)$ whether $T(\bar a)\neq S(\bar a)$
for some $A\in \mc{A}$ and $\bar a$ in $A$.

The decision problem sSAT$_{\mc{A}}$ 
is to decide for any  conjunction
$\phi_0(\bar x, \bar u)$ of basic equations 
whether there are $A\in \mc{A}$ and $\bar a$ in $A$
such that the members of $\bar a$ generate a non-trivial subalgebra of $A$
and 
such that $A \models  \exists \bar u.\,\phi_0(\bar a,\bar u)$.

\begin{fact}\lab{bsat}
 The decision problem
{\rm SAT}$_{\mc{A}}$  reduces in p-time  to {\rm sSAT}$_{\mc{A}}$.
In presence of constants $0\neq 1$ 
the problems are p-time equivalent.  
\end{fact}

\subsection{Feasibility}\lab{Feas}
We consider rings with basic operations
multiplication, addition, and subtraction and constants $0,1$.
For a  ring $R$,
the decision problem {\rm FEAS}$_{\mathbb{Z},R}$
is to decide, for any finite list of
terms
$p_i(\bar x)$,
 whether there is a  common zero
within $R$,
that is $\bar r \in R$ such that $R\models \bigwedge_i p_i(\bar r)=0$.
Thus, unless $R$ is a zero ring, this is 
just the problem SAT$_R$.

 According to \cite[Observation 1.9]{jacm},
this decision problem is p-time equivalent
to the analogous  decision  problem where the 
$p_i(\bar x)$ are
 multivariate polynomials
in non-commuting variables 
(commuting variables in case of commutative $R$) and with
integer coefficients,
each polynomial given as list of
monomials and coefficients.

The reduction is first from a list of terms $p_i(\bar x)$,
via Fact~\ref{ured},
to the
existentially quantified sentence obtained from the formula
$\bigwedge_i(\phi_i(\bar x,y_i,\bar z_i)\,\wedge\, y_i=0)$ where 
the $\phi_i(\bar  x,y_i,\bar z_i)$ are unnested terms,  each with separate auxiliary variables;
and then replacing in the latter each equation of the form
 $u=f(\bar v)$, $f$  a fundamental operation symbol,
by $u-f(\bar v)=0$.

\subsection{Retractive terms}
The hardness results to be derived, here, 
rely on 
the technique for reducing word problems to
free word problems, as  developed by Ralph Freese
\cite{freese} (in a much more sophisticated context).
Denote by $F(\mc{A};\pi)$ the $\mc{A}$-algebra
freely generated within ${\sf V}(\mc{A})$
 under the presentation $\pi=(\bar g,R)$,
 that is with  system $\bar g=(g_1,\ldots ,g_n)$
  of generator symbols   
and  set $R$ of relations.  Let  $\pi^+$ be obtained by extending $R$ to
a set  $R^+$ of relations in the same generator symbols $\bar g$.
Now,  consider
the  canonical homomorphism $\phi:F(\mc{A};\pi)\to F(\mc{A};\pi^+)$. 
$\phi$ is a retraction with  section homomorphism $\rho:F(\mc{A};\pi^+)\to F(\mc{A};\pi)$
if and only if there is a  system of terms $t_i(\bar x)$ such that
 for any
system $\bar a=(a_1,\ldots ,a_n)$ in 
any  $A \in \mc{A}$,  one has the following:
\begin{itemize}
\item
if $\bar a$ satisfies the relations $R$, then
 $(t_1(\bar a), \ldots ,t_n(\bar a))$ satisfies the relations $R^+$;
\item  if $\bar a$ satisfies the relations $R^+$, then
 $t_i(\bar a)=a_i$ for all $i$.
\end{itemize}
Here, $\rho$ and the $t_i$ are related by 
$\rho \phi g_i =t_i(\bar g)$ for $i=1,\ldots, n$. 
Such system of terms will be called \emph{retractive}
for $(\pi,\pi^+)$ within ${\sf V}(\mc{A})$.
 Retractions can be composed in steps:
 If $\pi'$ is obtained from $\pi^+$ by adding
relations and if
 the $s_i(\bar x)$
are retractive for $(\pi^+,\pi')$ within
${\sf V}(\mc{A})$ then so are the 
$s_i(t_1(\bar x), \ldots )$ for $(\pi,\pi')$.

As an illustration consider generator symbols $\bar g=(g_1,g_2,g_3,g_4)$,
$\mc{A}$ the variety of modular lattices, and $R$ the relations 
$g_1 \leq g_2$ and $g_3 \leq g_4$. Obtain $R^+$ by adding 
the relation $g_2\cdot g_4 \leq g_1 \cdot g_3$. 
Retractive terms for passing from $R$ to $R^+$
are given by $t_1= x_1+x_2\cdot x_4$, $t_2=x_2$, $t_3=x_3+x_2\cdot x_4$,
and $t_4=x_4$.

\subsection{Complemented modular lattices} 
We consider lattices $L$ with bounds $0,1$. The complexity results
to be derived are valid if the bounds are considered constants
in the signature as well if they are not.
$L$ is \emph{modular} if
$a\cap(b +c)= a\cap b +c$ for all $a,b,c$ with $c \leq a$ --
we write $a+b$  for joins, $a \cap b$ for meets and 
$\cap$ has binding priority  over $+$. 
Also, we use $\sum_i a_i$ and $\bigcap_i a_i$ for multiple joins and meets.
Elements 
$a_1, \ldots ,a_n$ 
of an interval $[u,v]$ are \emph{independent}
in  $[u,v]$  if $a_i\cap \sum_{j\neq i}a_j=u$
for all $i$; in this case, we write $\sum_i a_i=a_1 \oplus_u \ldots 
\oplus_u a_n$
and, if $u=0$, $\sum_i a_i=a_1\oplus \ldots \oplus a_n$.
A  lattice $L$ is \emph{complemented}
if for any $a \in L$ there is $b\in L$ such that $a\oplus b=1$.
A modular lattice  has \emph{height} $d$ if some (any)
maximal chain has $d+1$ elements.

Any complemented modular lattice $L$
of finite height is isomorphic  to a direct product of simple ones.
Thus, the decision problems SAT$_L$ and REF$_L$ break down
to the case of simple $L$.
 Up to isomorphism, the latter  are 
the subspace lattices of irreducible $(d-1)$-dimensional projective spaces.
With  exceptions for $d\leq 2$, these lattices
are isomorphic to lattices
${\sf L}(V_F)$ of all subspaces of
(left) $F$-vector spaces $V_F$,
$\dim V_F=d$, the case of most interest for us.
Thus, one  may read  all of the following in the context of  Linear Algebra.

\subsection{Frames}\lab{dframe}
Reducing arithmetic to modular (ortho-)lattices 
is most conveniently done via von Neumann frames
and their coordinate rings,
in particular if generators and relations 
 may be used
on the lattice side; cf. \cite{lip}.

A $d$-\emph{frame} $\bar a$ in a modular lattice 
 consists of elements $a_\bot \leq a_\top$
and $a_1,a_j,a_{1j} (1<j\leq d)$ in the interval $[a_\bot,a_\top]$
such that 
$a_\top =a_1 \oplus_{a_\bot} \ldots \oplus_{a_\bot} a_d$ and
$a_1 \oplus_{a_\bot} a_{1j}=a_j\oplus_{a_\bot} a_{1j}=a_1+a_j$ for $1<j\leq d$.

Given a $d$-frame $\bar a$
 in a modular lattice $L$ of height $d$,
if  $a_i>a_\bot$ for some  $i$, then $a_j>a_\bot$ for all $j$,
whence $[a_\bot,a_\top]$ has height at least $d$
and so $a_\bot=0$ and $a_\top =1$. 
In this case the frame is called 
 \emph{spanning};  the  $a_i$ are independent atoms
and $L$ is simple and complemented.
On the other hand, if $a_i=a_\bot$ for one $i$
then $a_j=a_\bot$ for all $j$ and $a_\bot=a_\top$;
such frame is called \emph{trivial}.
To summarize, a $d$-frame in
a height $d$ modular lattice is either spanning or trivial.
 Also, if $L$ is complemented and simple then  any atom $a_1$
is part of some (many, in general) spanning $d$-frame. Indeed,
the $a_i$ can be chosen so that $a_1, \ldots ,a_d$
are independent atoms and the $a_{1j}$
are  axes of perspectivity.

For $\dim V_F=d$,
the non-trivial $d$-frames in $\lt(V_F)$ have $a_\bot=0$, $a_\top=V$; 
and $\bar a$ is  a non-trivial $d$-frame       
if and only if there
is  a basis $e_1, \ldots ,e_d$ of $V_F$ such that
 $a_i=e_iF$, and $a_{1j}= (e_1-e_i)F$. 
Thus, the automorphism group of $\lt(V_F)$ acts
transitively on the set of such $d$-frames.
The following is well known and easy to prove.
\begin{fact}\lab{fred}
For any  $d$-frame $\bar a$ and $a_\bot \leq b_1 \leq a_1$,
the $b_i:=(b_1+a_{1j})\cap a_j$, $b_{1j:}=(b_1+b_j)\cap a_{1j}$,
 form a $d$-frame $\bar b$ with $b_\bot:=a_\bot$ and
$b_\top:=\sum_ib_i$ and such that
$\bar b=\bar a$ if $b_1=a_1$.
\end{fact}

Retractive terms for constructing
(equivalent variants of) $d$-frames in modular lattices have
been provided by Huhn \cite{huhn} and Freese \cite{ralph}; here,
using a list $\bar z$ of  variables for the elements of the $d$-frame to be obtained,  we 
denote such terms by
 $a_i(\bar z), a_{1j}(\bar z),
a_\bot(\bar z), a_\top(\bar z)$. The following is a special case of
\cite[Fact 5.2]{hz}.

\begin{lemma}\lab{discr}
For any $d$ there is a term $\delta_d(x,\bar z)$ 
such that for any spanning $d$-frame $\bar a$ in a  height $d$ modular lattice
$L$ and any $b \in L$ one has $\delta_d(b,\bar a)=1$ 
 if $b \neq 0$; moreover $\delta_d(0,\bar a)=0$.
\end{lemma}

\subsection{Coordinate ring}\lab{coo}
Fix $d\geq 3$. 
For each  operation symbol $+,-,\cdot,0,1$ 
in the signature of rings
and each term $t(\bar y)$
defining the associated operation,
 there is a lattice term $\tilde{t}(\bar y, \bar z)$
such that
each variable $y_i$ 
occurs only once in $\tilde{t}(\bar y, \bar z)$
and such that the following holds:
For    each 
 vector space $V_F$ and  $d$-frame $\bar a$ in $\lt(V_F)$
the set
\[R(\bar a)=\{ r \in {\sf L}(V_F)\mid r \oplus_{a_\bot}  a_2=  a_1+ a_2, \,r \cap a_2=a_\bot\}\] 
 becomes a ring, the \emph{coordinate ring} of $\bar a$,
with operations defined by $\tilde{t}(\bar y, \bar a)$ 
cf. \cite{ralph,freese}.
For example, for $t(y_1,y_2)= y_1 \cdot y_2$
one can choose 
$\tilde{t}(y_1,y_2,\bar z)= y_1 \otimes_{\bar z} y_2$ given as  
\[ [(y_2+z_{23})\cap (z_1+z_3)+ (y_1 +z_{13})\cap (z_2+z_3)] \cap (z_1+z_2).\]
Moreover, if $a_\bot=0$
then  there is a (unique) linear isomorphism 
 $\vep_{\bar a}: a_1 \to  a_2$ such that $ a_{12}=\{v-\vep_{\bar a} v\mid v \in  a_1\}$
as well as a ring  isomorphism $\omega_{\bar a}:\End( a_1) \to R(\bar a)$
given by $f\mapsto \{v-\vep_{\bar a} fv\mid v \in  a_1\}$.
 If $a_\bot=a_\top$ then $R(\bar a)$
is a zero ring.
If $\bar a$ is non-trivial and $\dim V_F=d$, then $R(\bar a)$ 
is isomorphic to $F$. 
We write $\otimes_{\bar a}$ and
$\ominus_{\bar a}$ for multiplication and subtraction. In particular,
the zero is $a_1$, the  unit is $ a_{12}$,
 and $r \in R(\bar a)$
is invertible if and only if $r \oplus_{a_\bot} a_1= a_1+ a_2$.
Due to the above mentioned unique occurrence of variables
one has the following.
\begin{fact}\lab{rtol}
With any ring term $p(\bar x)$ one associates
in p-time a lattice term $\tilde{p}(\bar x,\bar z)$
such that for any $d$-frame $\bar a$ in any $\lt(V_F)$ 
one has $\tilde{p}(\bar r,\bar a)=p(\bar r)$
for all $\bar r$ in $R(\bar a)$. 
\end{fact}

\section{$\mc{NP}$-hardness in 
complemented modular lattices}\lab{nph}
We will deal with simple complemented modular
lattices $L$ of finite height $d\geq 1$
(that is, $L$ isomorphic to the subspace lattice
of a $(d-1)$-dimensional irreducible projective space) and varieties generated by such lattices.
Here, the requirement that $\bar a$ generates a non-trivial
sublattice amounts to $\bar a$ being non-constant.

Observe that, for fixed finite $L$,
evaluating lattice  terms can be done in time
polynomial     in the  length of the  terms.
On the other hand, any $n$-generated sublattice of
$L$ of height at most  $2$ is isomorphic to $\bf 2$ or, for some $m \leq n$,
to the height $2$-lattice $M_m$ with $m$ atoms. 
Thus, if $L$ is finite or of height $2$ then both
SAT$_L$ and REF$_L$ are in $\mc{NP}$.

For lattices $L$, the case  where  the bounds are constants
and the case  where  they are not will not be distinguished
in notation, so that all results may be read both ways.
Observe in this context that the decision problems 
SAT and REF associated with the latter 
are subproblems of the former, obviously.

\subsection{Distributive lattices}\lab{dl}
\begin{theorem}\lab{ross}$($\cite{blon0, ross}$)$.  
For the $2$-element lattice $\bf 2$,
both {\rm SAT}$_{\bf 2}$ and {\rm REF}$_{\bf 2}$ 
are $\mc{NP}$-complete. 
In particular, the decision problem
for the equational theory of distributive lattices
is ${\rm co}\mc{NP}$-complete.
\end{theorem} 
\begin{proof}
We  include the proof, due to Ross Willard,  for $\mc{NP}$-hardness
of REF$_{\bf 2}$ in the case without constants.
The claim for SAT$_{\bf 2}$ follows
from Proposition~\ref{sattoSAT}, below.
Given a string $\bar x=(x_1,\ldots, x_n)$ of variables,
choose new variables $y_i,z$ and put $\bar y=(y_1,\ldots, y_n)$.
Define the lattice terms $\lambda_i(\bar x,\bar y,z)$
 by recursion: $\lambda_0=z$, $\lambda_{i+1}=\lambda_i \cap (x_i+y_i)+
x_i\cap y_i$.
Observe that in $\bf 2$ the polynomial function
$\lambda_n(\bar a,\bar b,z)$ is identity if $a_i\neq b_i$ for all $i$,
constant otherwise.  Also observe that  $\lambda_n$ 
has length linear
in $n$ and  a single occurrence of $z$.

Now, consider a boolean term $t(\bar x)$ 
without constants $0,1$ and in negation normal form. 
Replacing each occurrence of $x_i^\perp$
(where $^\perp$ denotes negation) by $y_i$, one obtains
 a lattice term $t^\#(\bar x,\bar y)$
such that $t(\bar a)= t^\#(\bar a,\bar b)$
holds in $\bf 2$ if $b_i=a_i^\perp$ for all $i$.
Consider the lattice equation $\varepsilon=\varepsilon(\bar x,\bar y)$ given as
\[\lambda_n(\bar x,\bar y,t^\#(\bar x,\bar y))=\lambda_n(\bar x,\bar y,x_1\cap y_1).  \] 
Thus, $\varepsilon$ has length linear in that of $t(\bar x)$.
We claim that $t(\bar a)=1$ for some $\bar a$ in $\bf 2$ 
if and only if $\varepsilon(\bar a,\bar b)$ fails for some $\bar a,\bar b$
in $\bf 2$.
Indeed, given  $\bar a,\bar b$ 
such that $a_i \neq b_i$, that is $b_i=a_i^\perp$, for all $i$,
one has $\lambda_n(\bar a, \bar b,t^\#(\bar a))=
t(\bar a)$ while $\lambda_n(\bar a,\bar b,a_1\cap b_1)= 
a_1\cap b_1=0$.
Thus, if there is $\bar a$ such that $t(\bar a)=1$ choose
$b_i=a_i^\perp$ for all $i$ to fail $\varepsilon$.
Conversely, assume $t(\bar a)\neq 1$, that is $t(\bar a)=0$ for all
$\bar a$. 
Then ${\bf 2} \models\varepsilon(\bar a,\bar b)$
for all $\bar b$: giving both sides value $0$  if $a_i\neq b_i$ for all $i$,
the constant value of $\lambda_n(\bar a,\bar b,z)$, otherwise.
\end{proof}

\subsection{Reduction of {\rm REF}$_L$ to {\rm SAT}$_L$}\lab{sS} 
\begin{proposition}\lab{sattoSAT}
For fixed  $d$, there is a p-time reduction of {\rm REF}$_L$ to
{\rm SAT}$_L$, uniform for all modular lattices $L$ 
of height $d$ and admitting a non-trivial $d$-frame.
\end{proposition}
\begin{proof}
The proof is the same with and without constants $0,1$.
 It suffices
to consider $t(\bar x),s(\bar x)$  such that
$\forall \bar x.\,t(\bar x) \leq s(\bar x)$ holds in all lattices
(replace $t$ by $t\cap s$ and $s$ by $t+s$).
Let $\phi(\bar z)$ be  the conjunction of equations defining
a $d$-frame. Now consider the following
conjunction $\psi(\bar x,\bar z)$ of lattice equations
(where $x\leq y$ means $x+y=y$)
\[\phi(\bar z) \;\wedge\; z_1\leq s(\bar x) 
\, \wedge\; z_1 \cap t(\bar x)=z_\bot\;\wedge
\;\bigwedge_{i=1}^n\; z_\bot \leq x_i\leq z_\top. \] 
Given $\bar b$ in $L$ such that
$t(\bar b)< s(\bar b)$, 
there is an atom $a_1 \leq  s(\bar b)$,
$a_1 \not\leq t(\bar b)$ (since $L$ is geometric)
and that  extends to a spanning $d$-frame $\bar a$; thus $L\models
\exists \bar x \exists \bar z.\;\psi(\bar x,\bar z)$.
Conversely, given non-constant  $\bar b,\bar a$ in $L$
such that $L \models \psi(\bar b,\bar a)$
  one must have $a_\bot \neq a_\top$,  whence
$\bar a$ is spanning; it follows that $t(\bar b) <s(\bar b)$.
\end{proof}

\subsection{$\mc{NP}$-hardness}
\begin{theorem}\lab{np} For any non-trivial modular lattice $L$
of height $d$ and admitting a nontrivial $d$-frame,
the decision problems
{\rm REF}$_L$ and {\rm SAT}$_L$ are $\mc{NP}$-hard;
these problems are $\mc{NP}$-complete if  $L$ is finite
or of height $2$.
\end{theorem}
\begin{proof}
To prove hardness, in view of Proposition~\ref{sattoSAT} it remains to consider
REF$_L$ for $L$ of fixed height $d$. Again, the proof is the same
whether the bounds $0,1$ are considered constants or not.
 Given a lattice term $t(\bar x)$,
let $\bar x'=(x_1',\ldots, x_n')$ where
 \[x_i':= \delta_d(z_\bot +z_\top \cap x_i,\,\bar z)\]
(from Lemma \ref{discr}) and put  
 $t'(\bar x,\bar z)=t(\bar x')$.  
Observe that for any assignment $\gamma$ for $\bar x,\bar z$ in $L$
one has 
either  $\gamma  \bar z$
trivial and $\gamma x_i'=\gamma z_\bot$ or
 $\gamma \bar z$ spanning and
$\gamma  x_i'\in \{0,1\}$; also, in the latter case,
$\gamma x_i'=\gamma x_i$ if $\gamma x_i \in \{0,1\}$.
 
Now, consider a second term $s(\bar x)$. Then there is an
assignment
 $\gamma$ for $\bar x$ in $\{0,1\}$
such that $t( \gamma \bar x)\neq s(\gamma \bar x)$
if and only if there is an assignment  
 $\gamma'$ for $\bar x\bar z$ in $L$ such that 
$s'(\gamma' \bar x \bar z)\neq t'(\gamma' \bar x \bar z)$.  
Namely, given $\gamma$ choose $\gamma' \bar z$ any spanning frame
and $\gamma' \bar x=\gamma \bar x$.
Conversely, given $\gamma'$ choose $\gamma  x_i := 
\delta_d(\gamma' z_\bot +\gamma' z_\top  \cap \gamma'x_i,\,\gamma \bar z)$.
We are done by $\mc{NP}$-hardness of REF$_{\bf 2}$, as 
shown in the proof of Theorem~\ref{ross}.
\end{proof}

\section{Relating feasibility to lattices}\lab{fsl}
Recall from Section~\ref{dl}
the remark on constants.

\begin{theorem} For fixed $3\leq d \leq \infty$
and $d$-dimensional  $F$-vector space $V_F$,
both {\rm SAT}$_{{\sf L}(V_F)}$ and 
 {\rm REF}$_{{\sf L}(V_F)}$ are p-time equivalent to 
{\rm FEAS}$_{\mathbb{Z},F}$.
\end{theorem}

\noindent
This follows from Proposition \ref{sattoSAT}, Fact~\ref{bsat}, and Subsections
\ref{SF} and \ref{Fs} below.

\subsection{Reduction of ${\rm sSAT}_ {\lt(V_F)}$
 to {\rm FEAS}$_{\mathbb{Z},F}$}\lab{SF}
The reduction of sSAT$_{\lt(V_F)}$ 
to FEAS$_{\mathbb{Z},F}$ via
non-determinism and BSS-machines is
contained in \cite[Proposition 2.1]{jacm}.
We sketch a direct  proof. One has  to consider
existentially quantified conjunctions of equations
of the form $x=y+z$, $x=y\cap z$, $x=0$,
$x=1$, and $x=y$. Using the partial order on $\lt(V_F)$,
each of these can be obtained as a conjunction of formulas
of the form  
$x \leq y$, $x \leq 0$, $1 \leq x$, $x \leq y+z$, and $y =z\cap u$.
Thus, each of the latter has to be translated into a 
feasibility condition.

 Dealing with $\lt(V_F)$ 
for fixed  $\dim V_F=d<\infty$, we may assume $V_F=F^d$.
Associate with $0$ and  $1$ the $d\times d$ zero and unit matrices
$O$ and  $I$,
and with  each variable  $\xi$
 a $d\times d$-matrix $\hat{\xi}$
of new variables to be interpreted in the ring
$F$; namely, 
  $\hat{\xi} \mapsto A$ 
corresponding to  $\xi \mapsto {\rm Span}(A)$
where ${\rm Span}(A)$ denotes the subspace of $F^d$
spanned by the columns of $A$.
In our translation,  quantifications over matrices of variables
always are with new variables (that is, specific
to the formula translated). 
Observe that $\Sp(A) \leq \Sp(B)$ if and only if
$A=BU$ for some matrix $U$. Thus, $x \leq y$ translates into
\[ \exists U.\,\hat{x}=\hat{y} U; \] 
and, if $0,1$ are considered constants, $x \leq 0$ into $\hat{x}=O$ and $1\leq x$ into $\exists U.\, I=\hat{x}U$.
Also observe that
$\Sp(A)\leq \Sp(B)+\Sp(C)$ if and only if 
$A=BU_1+CU_2$ for suitable matrices $U_1,U_2$.
Thus, we translate $x\leq y+z$  
 into
\[\exists U_1 \exists U_2\,
\hat{x}=\hat{y}U_1+\hat{z}U_2.\]  
 Meets are dealt with
using the idea underlying the Zassenhaus algorithm:
$y=z\cap u$ is translated into 
\[\exists X\exists Y \exists Z \exists U.\;
  \left(\begin{array}{cc} \hat{z}& \hat{u}\\
\hat{z}& 0 \end{array}\right)X
= 
 \left(\begin{array}{cc} Z& 0\\
U& \hat{y} \end{array}\right) \;\wedge\;
\left(\begin{array}{cc} Z& 0\\
U& \hat{y} \end{array}\right)Y
= 
 \left(\begin{array}{cc} \hat{z}& \hat{u}\\
\hat{z}& 0 \end{array}\right)
\]
with $2d\times 2d$-block matrices.
Observe that all these translations yield  conjunctions of
at most quadratic equations in non-commuting variables;
for  commutative  $F$, 
these can be converted in p-time into such
conjunctions  with commuting variables.

\subsection{Reduction of {\rm FEAS}$_{\mathbb{Z},F}$ to {\rm REF}$_L$}\lab{Fs}
This does not require the bounds $0,1$ to be considered constants.
Consider 
 $\vv$ where   $\dim V_F=d$  and recall that 
spanning $d$-frames $\bar a$  exist 
and that any non-spanning $d$-frame is trivial.
 Also observe that 
Fact~\ref{fred} allows one to pass from a
$d$-frame to a new one, given $a_\bot \leq b_1 \leq a_1$;
given  $a_\bot\leq b_2 \leq a_2$ first form $b_1:= a_1\cap (b_2+a_{12})$.
 We will apply this
reduction procedure to force relations.
In doing so, we will either have $\bar a=\bar b$ 
spanning or $\bar b$ trivial. Also, if $\bar a$ is trivial then
$\bar a=\bar b$.
Thus, we may think of a fixed but 
arbitrary spanning $d$-frame $\bar a^\ell$ to deal with in each step.

Given ring terms $p_k(\bar x)$,
we may consider the $x_i$ as lattice variables, too.
We start with the $d$-frame $\bar a^1$.
Given an assignment $r_i \in \vv$ for the $x_i$,
 put $r^1_i:=(a^1_\bot+ r_i)\cap (a^1_1+a^1_2)$ to achieve $a^1_\bot \leq r_i^1 
\leq a^1_1+a^1_2$.
Applying the reduction step with  $b_2:=a^1_2\cap \sum_i r^1_i$
one obtains a frame $\bar a^2$ and $r_i^2=r_i^1 \cap a^2_\top$ such that
 $r^2_i \cap a^2_2=a^2_\bot$. 
Forcing via $b_1:= a^2_1\cap \bigcap_i (r^2_i+a^2_2)$
one obtains $\bar a^3$ and $r_i^3$ such that
  $r^3_i\oplus_{a^3_\bot} a^3_2=a^3_1+a^3_2$, that is
$r_i^3 \in R(\bar a^3)$. 
Forcing via $b_1:=  a^3_1\cap \bigcap_k \tilde{p}_k(\bar r^3,\bar a^3)$
one obtains $\bar a^4$ such that
 $a^4_1=\tilde{p}_k(\bar r^4,\bar a^4)$ for all $k$,
that is $R(\bar a^4) \models p_k(\bar r^4)= 0$.
(This strongly relies on the fact that $\bar a=\bar b$ or $\bar b$ 
trivial. In general, reduction of frames 
will preserve only very special relations, as in the ingenious work of
Ralph Freese \cite{ralph,freese}.)

In order to implement the reduction of {\rm FEAS}$_{\mathbb{Z},F}$ to {\rm REF}$_L$, 
recall the retractive terms 
$\bar a=\bar a(\bar z)$ for $d$-frames 
to start with.
The forcing process described, above, can be captured by
a sequence of tuplets of lattice terms, obviously.
The result 
are (tuplets of)  lattice terms 
$\bar a^\#(\bar x,\bar z)$ and $\bar r^\#(\bar x,\bar z)$
such that, for any substitution $\gamma$ in $\vv$
for the variables  $\bar x$ and $\bar z$,
one has that   $\bar a^\#=\bar a^\#(\gamma \bar x, \gamma \bar z)$
is a $d$-frame   and that, with  $r_i^\#:=r_i^\#(\gamma \bar x, \gamma \bar z)$,
 either $\bar a^\#$ is trivial and $r_i^\#=a^\#_\bot$ for all $i$
or that $\bar a^\#$ is spanning and $\bar r^\#$ 
is a common zero of the $p_k$ in $R(\bar a^\#)$. 
Moreover, if  $\gamma \bar z$ is a spanning $d$-frame and 
$\gamma \bar x$ a common zero of the $p_k(\bar x)$ in $R(\gamma \bar z)$,
then $\bar a^\#(\gamma \bar x,\gamma \bar z)=\gamma\bar z$
and $\bar r^\#(\gamma \bar x,\gamma \bar z)=\gamma\bar x$.
Summarizing,
the $p_k$ have a common zero in $F$ if and only if there is
a substitution  $\bar \gamma$ in $\vv$ such that 
$a^\#_\bot(\gamma \bar x, \gamma \bar z)\neq
a^\#_\top(\gamma \bar x, \gamma \bar z)$.

\subsection{Varieties generated by complemented modular lattices}

\begin{fact}\lab{dr}
Given an atomic
 complemented modular lattice $L$
and  lattice terms $s(\bar x)$, $t(\bar x)$ such that
$L\models \forall  \bar x.\, s(\bar x) \leq t(\bar x)$.  
If $\exists \bar x.\, s(\bar x)<t(\bar x)$
holds in $L$
then it does so in a section $[0,u]$ of $L$
where the height of $u$ is at most the number 
of occurrences of variables in $t(\bar x)$.
\end{fact}

\begin{proof}
One shows by structural induction
 for any atom $p$ and term $t(\bar x)$:
if $p\leq t(\bar b)$,  then 
there is $\bar c$ such that $p\leq t(\bar c)$, $c_i\leq b_i$ for all $i$
and,   for     $u:=\sum_i c_i$, $[0,u]$ of height
at most the number of occurrences of variables in $t(\bar x)$.
Indeed, if $t=t_1+t_2$,  then there are atoms $p_j\leq t_j(\bar b)$
such that $p\leq p_1+p_2$. By inductive
hypothesis
there are $\bar c^j$ such that
 $p_j \leq t_j(\bar c^j)$, $c_i^j\leq b_i$, and 
 height of $\sum_i c_i^j$ at most the number of occurrences of variables in $t_j(\bar x)$.
Put $c_i:=c_i^1+c_i^2$. Similarly for  $t=t_1\cap t_2$, 
 given $p \leq t(\bar b)$, by inductive hypothesis there are
$\bar c^j \leq \bar b$ such that   
  $p\leq t_j(\bar c^j)$ and one has $p\leq t(\bar c)$
where $c_i= c^1_i \cap c^2_i$. 
Now, if $s(\bar b) <t(\bar b)$, choose an atom
$p\leq t(\bar b)$, $p\not\leq s(\bar b)$.
Then $p\leq t(\bar c)$ but  $p\not\leq s(\bar c)$.  
\end{proof}

For  $c$, $c$ prime or $0$,
let  $\mc{V}_c$ be the smallest lattice variety
containing all $\lt(V_F)$ over division rings of characteristic $p$.
Observe that for a division ring $F'$ embedded into $F$,
$\lt(V_F)$ is a sublattice of $\lt(V_{F'})$
if $V$ is considered a vector space over $F'$
and that $\lt(W_{F'})$ is embedded into $\lt(V_F)$ if $\dim W_{F'}=\dim V_F$. 
It follows from Fact~\ref{dr} that, for each $F$ of characteristic $c$,
  $\mc{V}_c$ is generated by  any class of $\lt(V_F)$
where $\dim V_F$ is finite and  unbounded.
For each $c$, the equational theory of  $\mc{V}_c$
is decidable \cite{hh,hc}.
An upper bound on the complexity is given by the following 
(though, it remains open to establish a lower bound, say $\mc{NP}$).

\begin{corollary}
  Given a field $F$ of characteristic $c$,
there is a  p-time reduction of ${\rm REF}_{\mc{V}_c}$
to {\rm FEAS}$_{\mathbb{Z},F}$.
In particular, {\rm REF}$_{\mc{V}_c}$ is in $\mc{NP}$
for any prime $c$.
\end{corollary}
\begin{proof}
By Fact \ref{dr}, if an equation fails  in some member of $\mc{V}_c$,
then it does so in  $\lt(V_{F})$
where $\dim V_{F}$ is  the number of occurrences of variables
in the equation. 
Now apply  Subsections \ref{sS} and \ref{SF}
for reduction to  {\rm FEAS}$_{\mathbb{Z},F}$.
\end{proof}

\section{Preliminaries: Part II}\lab{p2}

\subsection{Inverses in coordinate rings} \lab{inv}

Continuing with Subsections~\ref{dframe} and \ref{coo}
we consider the lattice $\lt(V_F)$ of subspaces 
of an $F$-vector space and a $3$-frame $\bar a$
such that $a_\bot=0$ and $a_\top=V$. 
We extend $\bar a$ 
adding the elements  $a_{23} =(a_2+a_3)\cap (a_{12}+a_{13})$ and
 $a_{ji}=a_{ij}$ for $i<j$. Observe that, by modularity, one has
\[(+)\quad (a_i +a_j)\cap (a_{ik}+a_{kj})= a_{ij} 
\mbox{ for } \{i,j,k\}=\{1,2,3\}.\]
Thus, one obtains a 
\emph{normalized frame of order} $3$ in the sense of von Neumann.
 Also,
we introduce corresponding variables
$z_{23}$ and  $z_{ji}$ and the \emph{perspectivity terms} 
 \[\pi_{ijk}^{\bar z}(x):=(x+z_{jk})\cap(z_i+z_k)\]
which define the isomorphism $u \mapsto \pi_{ijk}^{\bar a}(u)$
 of $[0,a_i+a_j]$ onto $[0,a_i+a_k]$; in particular $a_i \mapsto a_i$,
$a_j \mapsto a_k$, and $a_{ij} \mapsto a_{ik}$.

For given $i\neq j$ and endomorphism  $f$ of $a_1$
, define $\Gamma_{ij}^{\bar a}(f)=\{ v-f(v) \mid v \in a_i\}$,
 the (negative) \emph{graph} of $f$;
this
establishes a 1-1-correspondence between
  linear maps  $f:a_i\to a_j$ and subspaces $U$ such that 
$U\oplus a_j=a_i+a_j$. Invertible $f$ are characterized  by the additional
condition $U\oplus a_i=a_i+a_j$;
here, 
$\Gamma_{ji}^{\bar a}(f^{-1})=\Gamma_{ij}^{\bar a}(f)$.

 In particular, one has the linear
 isomorphisms
$\vep_{ij}^{\bar a}: a_i \to a_j$ 
 such that  $\Gamma_{ij}^{\bar a}(\vep^{\bar a}_{ij})=a_{ij}$.
From $a_{ji}=a_{ij}$ it follows that $\vep^{\bar a}_{ji}=
\vep^{\bar a}_{ij}$ (observe that $\vep^{\bar a}_{12}=\vep_{\bar a}$ and
$\vep^{\bar a}_{21}=\vep_{\bar a}^{-1}$). 
For pairwise distinct $i,j,k$ and linear maps
$f:a_i \to a_j$, $g:a_j\to a_k$ one has
\[(++) \quad \Gamma^{\bar a}_{ik}(g  f)=     
(\Gamma^{\bar a}_{ij}(f) + \Gamma^{\bar a}_{jk}(g))\cap (a_i +a_k). \]
In particular,  from (+) it follows that
$\vep^{\bar a}_{jk}  \vep^{\bar a}_{ij}= \vep^{\bar a}_{ik}$.
 Define the lattice term
\[ s(x,\bar z):=\pi_{231}^{\bar z}\pi_{312}^{\bar z}\pi^{\bar z}_{123}(x).\]
 
\begin{fact} \lab{inverse}
For $\lt(V_F)$ and $\bar a$ as above and
 $f \in \End(a_1) $
one has
\[ s(\omega_{\bar a}(f)), \bar a)= \{ f\vep^{\bar a}_{21}(w) -w\mid w \in a_2\}.
 \]  
In particular, $ s(\omega_{\bar a}(f), \bar a)= \omega_{\bar a}(f^{-1})$
if $f$ is invertible.
\end{fact} 
\begin{proof}
Observe  that $\omega_{\bar a}(f)= \Gamma^{\bar a}_{12}(\vep^{\bar a}_{12} f)$. Applying (++) three times we get
$s(\omega_{\bar a}(f),\bar a))= \Gamma^{\bar a}_{21}(g)$
where 
$g=  (\vep^{\bar a}_{21}((\vep^{\bar a}_{12} f) \vep^{\bar a}_{23})) \vep^{\bar a}_{31}=f  \vep^{\bar a}_{21}$.
If $f$ is invertible, then $\vep^{\bar a}_{12}  f^{-1}$ 
is an isomorphism of $a_1$ onto $a_2$ and with $w=\vep^{\bar a}_{12} f^{-1}(v)$ one gets
$\{ f\vep^{\bar a}_{21}(w) -w\mid w \in a_2\}
=\omega_{\bar a}(f^{-1}) $. 
\end{proof}

\subsection{Modular ortholattices}
An \emph{ortholattice} is a lattice $L$
with bounds $0,1$ as constants in the signature as well as a map,
called \emph{orthocomplementation},
$a \mapsto a^\perp$ such that for all $a,b \in L$
\[ a\leq b \;\Leftrightarrow\;b^\perp \leq a^\perp, \quad
(a^\perp)^\perp =a,\quad a\oplus a^\perp =1.\]
We write $\sum_i a _i=a_1\oplus^\perp \ldots \ldots \oplus^\perp  a_n$ 
iff $a_i\leq a_j^\perp$ for
all $i\neq j$.
A  \emph{MOL} is a modular ortholattice.
Let $\mc{H}$ denote the class of  all
 finite dimensional real or complex Hilbert spaces and
$\mc{H}_{\mathbb{R}}$ the class of those  $H\in \mc{H}$ which are over $\mathbb{R}$. 
The lattice  of all linear subspaces of a given $H \in \mc{H}$ 
is a MOL $\lt(H)$ with $X^\perp$ the orthogonal complement of $X$
w.r.t. the  inner product.

In a modular ortholattice $L$,  every  section $[0,u]$
is again a MOL with $x \mapsto x^\perp \cap u$.
Moreover, $[0,u] \in {\sf HS}(L)$ via the homomorphism
$x \mapsto x \cap u$  defined on the sub-ortholattice
$[0,u] \cup [u^\perp,1]$ of $L$. 
For a subspace $U$ of $H$, the ortholattice so  obtained
 is  $\lt(U)$.  

\begin{fact}\lab{ol}
Within the class of MOLs
any conjunction of identities is equivalent to a single
identity of the form $t=0$.
\end{fact}
Indeed,
 any identity $s=t$ is equivalent to 
 $(s+t)\cap (s\cap t)^\perp=0$
and $\bigwedge_i t_i=0$ is equivalent to 
$\sum_i t_i=0$.
Thus, for a class 
$\mc{C}$ of MOLs, REF$_\mc{C}$
and SAT$_\mc{C}$ both amount to the following:
 Given a term $t(\bar x)$,
are  there non-trivial $L\in \mc{C}$ and an assignment $\bar a$ in $L$
such that $t(\bar a) \neq 0$ (respectively,
 $t(\bar a)=1$)? Similarly, for 
uREF$_\mc{C}$
and sSAT$_\mc{C}$.

\subsection{Dimension bounds in REF$_{\lh}$}\lab{dim}

\begin{lemma}\lab{ldim}
Given an unnested ortholattice term 
$T(\bar x)$ and  $H \in \hr$ such that
 $\lt(H) \models \exists \bar x.\,T(\bar x)\neq 0$,
there is   $H' \in \hr$, 
$\dim H'= o(T(\bar x))$ such that 
 $\lt(H') \models \exists \bar x.\,T(\bar x)\neq 0$.
Analogously, in the complex case.
\end{lemma}
\begin{proof}
The analogue for terms follows from \cite[Lemma 2.2]{neu}
and the fact $\lt(H)\in {\sf HS}(\lt(H'))$
for any extension  $H'$ of $H$. In view of Fact~\ref{ured}
this proves the Lemma.
\end{proof}

\subsection{Orthogonal frames}\lab{oframe}
Recall Subsection~\ref{dframe}. 
A $d$-frame in a MOL
is \emph{orthogonal} if $a_\bot =0$ and
$\sum_i a_i=\bigoplus_i^\perp a_i$. 
For a Hilbert space $H=\bigoplus^\perp_i H_i$
with isomorphisms  $\vep_j:H_1 \to H_j$ one obtains an 
orthogonal $d$-frame in $\lt(H)$ with $a_i=H_i$
and $a_{1j}=\{ x -\vep_j x\mid x \in H_1\}$;
and all orthogonal $d$-frames in $\lt(H)$  with $a_\top=H$, $H \in \mc{H}$,
arise this way.
Observe that $\bar a$ is an orthogonal $d$-frame if and only  it
is so in the section $[0,a_\top]$. 
Fact~\ref{fred} reads now as follows.
\begin{fact}\lab{red}
For any orthogonal $d$-frame $\bar a$
in a modular ortholattice $L$ and $b_1 \leq a_1$,
the $b_i:=(b_1+a_{1j})\cap a_j$, $b_{1j}:=(b_1+b_j)\cap a_{1j}$,
 form an orthogonal $d$-frame $\bar b$
with  $b_\bot:=a_\bot=0$ and $b_\top:=\sum_ib_i$
 and such that 
$\bar b=\bar a$ if $b_1=a_1$.
Moreover, if $L=\lt(H)$ and if $b_1$ is invariant for  a given $f\in \End(a_1)$,
then $\omega_{\bar b} f= (b_1+b_2)\cap \omega_{\bar a} f$.
\end{fact}
Existence of retractive terms for orthogonal $d$-frames
within the variety of MOLs 
is due to Mayet and  Roddy 
\cite{MR}. More easily it 
 is seen as follows. We will need only the case $d=3$
and $L=\lt(H)$.
 Given
$\bar a$ in the modular ortholattice $L$, let
 $u:=\bigcap_{i\neq j}(a_i+a_j)$ 
and $a_i^1:=u+a_i$. Then $a_1^1,a_2^1,a_3^1$ is independent 
in $[u,1]$.
 Thus, ${a_1^1}^\perp,{a_2^1}^\perp,{a_3^1}^\perp$
is dually independent in $[0,u^\perp]$ and
defining $a_i^2:={a_j^1}^\perp \cap {a_k^1}^\perp$,
$\{i,j,k\}=\{1,2,3\}$, one has $a_\top^2:=u^\perp= a_1^2\oplus^\perp
a_2^2\oplus^\perp a_3^2$.
 Put $a_{12}^2:= a_{12}\cap (a_1^2+a_2^2)$. 
For 
 $\{i,j\}=\{1,2\}$ put $c_i=a_i\cap (a_j+a_{12})$ and
$d_i=a_i \cap a_{12}$ to obtain the sublattice
in  Figure 1.
Now, with  $a^3_i=c_i\cap d_i^\perp$ for $i=1,2$ 
one has  $a_1^3+a_2^3$  a complement of $d_1+d_2$ in $[0,c_1+c_2]$
and with $a_{12}^3=(a_1^3+a_2^3)\cap a_{12}$ 
it follows that $a_i^3 \oplus a_{12}^3=a_1^3+a_2^3$.
Put  $a^3_\top:= a_1^3\oplus^\perp
a_2^3\oplus^\perp a_3^3$. Similarly, one obtains $\bar a^4$ such that
$a^4_\top:= a_1^4\oplus^\perp
a_2^4\oplus^\perp a_3^4$ and $a^4_{13}\oplus a^4_j= a^4_1+a^4_3$
for $j=1,3$. Finally, put $a^5_j:=a^4_j$ for $j=1,3$,
$a^5_{13}:=a^4_{13}$, $a^5_2:= (a^5_1 +a^4_{12})\cap a^4_2$,
and $a^5_{12}:= a^4_{12} \cap (a^5_1+a^5_2)$
to obtain an orthogonal $3$-frame ${\bar a}^5$.

\begin{figure} 
\setlength{\unitlength}{8mm}
\caption{}   \label{pq}
\begin{picture}(8.5,7)

\put(0,3){\circle*{0.2}} 
\put(-.5,2.7){$c_1$} 
\put(6,3){\circle*{0.2}}  
\put(6.2,2.7){$c_2$}  
\put(3,3){\line(0,1){3}}
\put(3,4.5){\circle*{0.2}}  
\put(3.3,4.5){$a_{12}^2$}
\put(1.5,1.5){\circle*{0.2}}  
\put(0.8,1.3){$d_1$}  
\put(3,3){\circle*{0.2}}
\put(4.5,1.5){\circle*{0.2}}  
\put(5,1.3){$d_2$}  
\put(3,0){\circle*{0.2}}
\put(2.5,-.4){$0$}
\put(1.5,4.5){\circle*{0.2}}
\put(4.5,4.5){\circle*{0.2}}
\put(3,6){\circle*{0.2}}
\put(0,3){\line(1,1){3}}
\put(0,3){\line(1,-1){3}} 
\put(1.5,1.5){\line(1,1){3}} 
\put(3,0){\line(1,1){3}} 
\put(1.5,4.5){\line(1,-1){3}}
\put(3,6){\line(1,-1){3}}

\end{picture}
\end{figure}

\subsection{Coordinate ring}
Consider $d \geq 3$ and  an orthogonal $d$-frame $\bar a$ in $\lt(H)$,
$H \in \mc{H}$.  Recall the coordinate ring $R(\bar a)$
from Subsection~\ref{coo}.
In the ortholattice setting,  elements of the coordinate ring
can be forced via the following where
$\#(x,y,z)$ is  a term  defining the relative orthocomplement
of $z$ in the interval $[x,y]$,
\[ \#(x,y,z):= x+ y\cap z^\perp.\]
Observe that, by modularity, $\#(x,y,z)= (z^\perp+x)\cap y$ if  $x \leq y$.
Now, define 
\[ \#"(x,y,z):=\#(z\cap (x\cap z)^\perp, y, x+ z) 
\]

\begin{lemma}\lab{var}
 In a MOL, if $a,c \leq b$, then $\#"(a,b,c)\oplus a=b$. Moreover, if $b=c\oplus a$
then $\#"(a,b,c)=c$.
\end{lemma}
\begin{proof}  See Figure 2. 
Let $c'=c\cap  (a \cap c)^\perp$ and $c"= \#(c',b, a+c)$. 
From $c=c'\oplus a \cap c$  follows $a+c=c'\oplus a$.
Now, $c"$ is a complement of $a+c$ in $[c',b]$, whence of
$a$ in $[0,b]$. Clearly, if $b=c\oplus a$, then

$c=c'=c"$. 
\end{proof}

\begin{figure} 
\setlength{\unitlength}{8mm}
\caption{}   \label{pq2}
\begin{picture}(8.5,7)

\put(0,3){\circle*{0.2}} 
\put(-.7,2.7){$c"$} 
\put(6,3){\circle*{0.2}}  
\put(6.2,2.7){$a$}  
\put(1.5,1.5){\circle*{0.2}}  
\put(0.8,1.3){$c'$}  
\put(3,3){\circle*{0.2}}
\put(4.5,1.5){\circle*{0.2}}  
\put(5,1.3){$a\cap c$}  
\put(3,0){\circle*{0.2}}
\put(2.5,-.4){$0$}
\put(1.5,4.5){\circle*{0.2}}
\put(4.5,4.5){\circle*{0.2}}
\put(5,4.4){$a+c=a+c'$} 
\put(3,6){\circle*{0.2}}
\put(0,3){\line(1,1){3}}
\put(0,3){\line(1,-1){3}} 
\put(1.5,1.5){\line(1,1){3}} 
\put(3,0){\line(1,1){3}} 
\put(1.5,4.5){\line(1,-1){3}}
\put(3,6){\line(1,-1){3}}
\put(3.3,6){$b$}
\put(3.3,2.9){$c$}

\end{picture}
\end{figure}

\subsection{Orthonormal frames}\lab{onframe}
Recall that for any subspaces $U,W$ of $H\in \mc{H}$
and any linear map $f:U\to W$ there
is a unique linear map $g:W\to U$
such that $\langle x \mid g(y)\rangle =\langle f(x)\mid y \rangle$
for all $x \in U$, $y \in W$. $g$ is called  the \emph{adjoint} of $f$
and in case $U=W=H$ we write $g=f^*$.
An isomorphism $f:U \to W$ is an isometry if and only if
the inverse is the adjoint.
If $U \subseteq W^\perp$ then $g$ is the adjoint of $f$  
if  and only if  $\{v+fv\mid v \in a_1\}$ is orthogonal to 
$\{w -gw\mid w \in a_2\}$.

\begin{lemma}\lab{isometry}
Given an orthogonal $3$-frame $\bar a$ in $\lt(H)$, $H\in \mc{H}$
one has $\vep_{ \bar a}$ an isometry if and only if
\[(*)\quad   a_{12}^\perp \cap (a_1+a_2)= a_1\ominus_{\bar a}a_{12}.  \]
\end{lemma}
\begin{proof}
Observe that $a_1\ominus_{\bar a}a_{12}= \omega_{\bar a}(-\id_{a_1})$.
Thus, (*) holds if and only if $\omega_{\bar a}(-\id_{a_1}) \leq
a_{12}^\perp$ if and only if $\langle v+\vep_{\bar a}(v) \mid
w -\vep_{\bar a}^{-1}(w) \rangle =0$ 
for all $v \in a_1$ and $w \in a_2$ 
if and only if 
 $\langle v \mid
\vep_{\bar a}^{-1}(w) \rangle =
\langle \vep_{\bar a}(v) \mid
w  \rangle $ for all $v \in a_1$ and $w \in a_2$
if and only if $\vep_{\bar a}^{-1}$ 
and $\vep_{\bar a}$ are  adjoints of each other.
\end{proof}
An orthogonal $3$-frame
 $\bar a$ is called an  \emph{ON-$3$-frame} if and only if (*) holds. 
In particular, 
any
$H_1\in \mc{H}$ occurs as $a_1$ for some ON-$3$-frame
$\bar a$ in  some $H$, $\dim H= 3 \dim H_1$.
Also, the automorphism group of $\lt(H)$ acts transitively 
of the set of ON-$3$-frames.

In order to construct retractive terms for passing
from orthogonal $3$-frames to ON-$3$-frames we 
use Fact~\ref{red}. Here,
 we have $f=-{\sf id}$ and put
$b_1:= (a_{12}^\perp \cap \omega_{\bar a} f +a_2)\cap a_1$.
Summarizing subsection~\ref{oframe} and \ref{onframe}  and Lemma~\ref{var},
 we have retractive terms for passing
from $\bar a$ and $\bar r$ 
 with no relations to ON-$3$-frames
 and   elements of the coordinate ring. More precisely,
\begin{lemma}\lab{retr} 
Given variables $\bar z=(z_\bot,z_\top,z_1,z_2.z_3,z_{12},z_{13})$
and $x_i$ there are ortholattice terms
$\bar a(\bar z)=
(a_\bot,a_\top,a_1,a_2,a_3,a_{12},a_{13})(\bar z)$ 
and  $r_i(x_i,\bar z)$,
such that for any $H \in \mc{H}$
and substitution $\bar z \to \bar u$  and $x_i \to v_i$
in $\lt(H)$, 
one has $\bar a(\bar u)$ an ON-$3$-frame and $r_i(v_i,\bar u)
\in R(\bar a)$;
moreover, if $\bar u$ is an ON-$3$-frame and $v_i\in R(\bar u)$
 then $\bar a(\bar u)= \bar u$ and $r_i(v_i)=v_i$.
\end{lemma}

\subsection{Capturing adjoints in coordinate $*$-rings}\lab{adjoints}
The endomorphism ring $\End(H)$ of $H \in \mc{H}$
is also a  $*$-\emph{ring}: A ring with involution,
    namely the involution 
$f \mapsto f^*$, the adjoint of $f$.
Given subspaces $U$, $W$ with orthogonal projections
$\pi_U$, $\pi_W$  and a linear map $f:U\to W$,
one has $\pi_Uf^*|_W$ the adjoint of $f$ where $f^*$ refers to 
$f$ considered a map $f:U\to H$.
In particular, $\End(U)$ becomes a $*$-ring 
$\End^*(U)$ where $f \mapsto \pi_U f^*|_U$. 
In order to capture the involution in the coordinate ring
of an ON-$3$-frame we use the following ortholattice term
(where  $s(x,\bar z)$ is from Subsection~\ref{inv}):
\[ x^{\dagger_{\bar z}}:= s((z_1 \ominus_{\bar z} x)^\perp \cap (z_1+z_2)),\bar z). \]
\begin{lemma}\lab{adj}
Given $H \in \mc{H}$, an ON-$3$-frame
$\bar a$ in $\lt(H)$, $f \in \End(a_1)$, and $r=\omega_{\bar a}(f)$,
one has $r^{\dagger_{\bar a}} =\omega_{\bar a}(g)$ where
$g=\pi_U f^*|_U$ is the adjoint of $f$ in 
the Hilbert space $U=a_1$.  In particular, $\omega_{\bar a}$
is an isomorphism of  $*$-rings $\End^*(a_1)$ 
and $R(\bar a)$, the latter with involution $r \mapsto 
r^{\dagger_{\bar a}}$.
\end{lemma}
\begin{proof}
Write $U=a_1$ and  consider $f,g \in \End(U)$.
Then, as observed above, $g$ is adjoint to $f$ 
if and only if $g=\pi_Uf^*|_U$. 
Now,  one has 
$\pi_Uf^*\vep_{\bar a}^{-1} =\pi_U f^*\pi_U\vep_{\bar a}^{-1}
= \pi_Uf^*\pi_U(\vep_{\bar a})^* =\pi_U (\vep_{\bar a}\pi_Uf)^* =\pi_U(\vep_{\bar a}f)^* $.
Thus, $g$ is adjoint to $f$ if and only if $g\vep_{\bar a}^{-1}$ is adjoint to
$\vep_{\bar a}f$ if and only if 
$\Gamma^{\bar a}_{12} (-\vep_{\bar a}f)$ is orthogonal to
$\Gamma^{\bar a}_{21}(g \vep_{\bar a}^{-1})$,
 that is, $\Gamma^{\bar a}_{21}(g \vep_{\bar a}^{-1})=
\Gamma^{\bar a}_{12} (-\vep_{\bar a}f)^\perp \cap(a_1+a_2)$.
The latter is
equal to  $r':=(a_1\ominus_{\bar a} r)^\perp\cap (a_1+a_2)$.
In view of Fact~\ref{inverse}, $g$ is adjoint to $f$ if
and only if 
 $\omega_{\bar a}(g)= s(r',\bar a)= \omega_{\bar a}(g)$.
\end{proof}
In the sequel, we will consider adjoints only 
within the space $a_1$ where $\bar a$ is an ON-$3$-frame in some $\lt(H)$,
$H\in \mc{H}$, and use $f^*$ to denote the adjoint of $f \in \End(a_1)$  
within  $a_1$.

\section{Complexity of the equational theory of $\mc{L}$}\lab{hilb}
Recall that $\mc{L}$
is the class of subspace ortholattices $\lt(H)$,
$H \in \mc{H}$, where $\mc{H}$ is the class of  all finite dimensional real or complex Hilbert spaces.
Our main result is the following.

\begin{theorem}\lab{thm}
The decision problems
{\rm REF}$_{\lh}$, {\rm uREF}$_{\lh}$, and 
   {\rm FEAS}$_{\mathbb{Z},\mathbb{R}}$ are p-time equivalent to each other. 
\end{theorem}

\subsection{Outline}\lab{out}
By  
Fact~\ref{neu} below, 
 ${\sf V}(\lh)$ is generated by the 
$\la(H)$, $H \in \hr$, 
where $\hr$ 
denotes the class of all finite dimensional real Hilbert spaces.
Thus, to reduce to REF$_\lh$,
 our task is the following
(recall Fact~\ref{ol} and  that for  $\bar r \in\mathbb{R}$
one has $p_i(\bar r)=0$ for $i=1,\ldots ,m$
if and only if $\sum_i p_i(\bar r)^2=0$).
\begin{quote} 
($*$) Given a term  $p(\bar x)$ in the language of rings, 
construct (in p-time) a conjunction $\phi_p(\bar x,\bar z)$
of ortholattice identities 
such that $p(\bar r)\neq 0$ for all $\bar r$ from  $\mathbb{R}$
 if and only if ${\sf L}(H)\models \forall \bar x \forall \bar z.\,\phi_p(
\bar x,\bar z)$ for all  $H\in \hr$.
\end{quote}
In  \cite[Proposition 4.9]{jacm}  the commutativity relations required by the 
Spectral Theorem have been encoded in a conjunction
of equations (to be viewed as a system of generators and relations)
to prove that FEAS$_{\mathbb{Z},\mathbb{R}}$
reduces to SAT$_\lh$.
Though, dealing with REF$_\lh$,
we have to force such relations by Ralph Freese's
technique of retractive terms. The first step
is to construct  elements in a  coordinate $*$-ring of an ON-$3$-frame, 
via  Lemma~\ref{retr}.
This is used in Subsection~\ref{thm2}
to reduce  from    {\fs}  to REF$_{\lh}$. 
 The reduction from REF$_\lh$
is in \cite[Theorem 4.4]{jacm} via non-deterministic
BSS-machines. We give a direct proof: 
In
(Subsections~\ref{feas1}--\ref{feas}), we reduce  in fixed dimension via
SAT$_{\lt(H)}$ to {\fs}~; to combine these to the required reduction,
 we use the dimension
bound (from \cite{neu})  for failure of equations.
For later use, we describe these reductions for uREF$_{\mc{L}}$ 
in place of REF$_{\mc{L}}$.

Lemma~\ref{lu1}, below, gives the
reduction  from  FEAS$_{\mathbb{Z},\mathbb{R}}$ to
REF$_{\rh}$, $\rh$ the class of the $\End(H)$
(considered as $*$-rings with pseudo-inversion),
 $H \in \mc{H}$. 
Not requiring retractive terms,
the proof is much simpler and might be read first. 
Though, to have the   reduction from REF$_{\rh}$ (in Subsection~\ref{uref})
in polynomial-time, it 
is only to uREF$_{\lh}$; this is due to multiple occurrence
of variables in the  term translating, to
the language of $\lh$, the
fundamental operation of pseudo-inversion.
It remains open whether there is  a direct p-time reduction of
REF$_{\rh}$ to REF$_{\lh}$.

\subsection{Varieties generated by ortholattices $\lt(H)$}

\begin{fact}\lab{neu}
For every   $\mc{C} \subseteq 
 \lh$,  
the variety ${\sf V}(\mc{C})$
either is generated by one or two members of  $\mc{C}$
or is  equal to $\lh$
and is generated by any family in 
$ \lh$
  having
unbounded dimensions.  
\end{fact}
We mention that 
${\sf V}(\lh)$ contains all projection ortholattices
of finite Rickart $C^*$-algebras \cite{hs}.

\begin{proof}
Recall that for any $d$ there is, up to (isometric) isomorphism,
just one 
 real (respectively complex) 
Hilbert space of dimension $d$.
Also, if $H_1,H_2$ are both real (respectively complex)
 and $d_1 \leq d_2$ then ${\sf L}(H_1)$
is a homomorphic image of a sub-ortholattice of ${\sf L}(H_2)$,
namely ${\sf L}(U) \times {\sf L}(U^\perp)$
embeds into ${\sf L}(H_2)$ where $U \in {\sf L}(H_2)$
with $\dim U=d_1$.
Recall, finally, that ${\sf L}(H)$, $H\in \hr$,
 embeds into the  subspace ortholattice
of the complex Hilbert space $\mathbb{C }\otimes_{\mathbb{R}}H$; and
recall  that
for $H$ over $\mathbb{C}$,
considering $H$ as the complexification 
of the $\mathbb{R}$-vector space $H_0$,
${\sf L}(H)$ embeds into
 ${\sf L}(H_0)$  where $H_0$
is endowed with the real part of the scalar product on $H$ .
\end{proof}

\subsection{Reducing {\fs}  to {\rm REF}$_{\lh}$}\lab{thm2}
We have to prove (*) from Subsection~\ref{out}. In view of  Lemma~\ref{retr}
we may continue from an ON-$3$-frame $\bar a$ and 
 $r_1,\ldots ,r_n\in R(\bar a)$.
Now, recall that for $f \in \End(a_1)$
one has 
\[\ker f = a_1\cap \omega_{\bar a} f. \] 
Thus, with $r_i:=\omega_{\bar a} f_i$, $i=1,\ldots ,n$, 
and $X:=\{r_i,r_i^{\dagger_{\bar a}}\mid i=1,\ldots ,n\}$
put
\[u:=k(\bar r,\bar a):= \bigcap_{r,s \in X} a_1\cap(r\otimes_{\bar a} s
\ominus_{\bar a} s \otimes_{\bar a} r).\]
Denoting by $c$ the set of all vectors $v\in a_1$ on which 
any two  of  $f_i,f_i^*,f_j,f_j^*$ commute,
we have $c=u$.
Let (in infix notation)  $x\tilde{+}_{\bar z}y$ 
denote the lattice term
defining addition in $R(\bar a)$ 
and $\widetilde{2^{-1}}_{\bar z}$ the one for the inverse
of $a_{12}\tilde{+}_{\bar a}a_{12}$ (cf. Fact~\ref{inverse}). 
Define
\[s_i:= \widetilde{2^{-1}}_{\bar a}\bigl((r_i\cap (u+a_2)  + u^\perp) \tilde{+}_{\bar a} (r_i\cap (u+a_2)  + u^\perp)^{\dagger_{\bar a}}\bigr),  \] that is
$s_i =\omega_{\bar a} \frac{1}{2}(\hat{f}_i+\hat{f}^*_i) $ where $\hat{f}_i$
is $f_i$ on $u$ and $0$ on $u^\perp$. 
In particular, the $\hat{f}_i$ are self-adjoint and commute on $a_1$. 
Moreover, if the $r_i$, respectively $f_i$, commute and are self-adjoint then $a_1=k(\bar r,\bar a)$,
$f_i=\hat{f}_i$ and $r_i=s_i$.

Of course, these definitions work uniformly for all $H \in \mc{H}$ and 
ON-$3$-frames $\bar a$ and $r_i\in R(\bar a)$ in $\lt(H)$. 
Thus, we have achieved  terms $t_j(\bar z)$, $t_{1j}(\bar z)$, and 
$s_i(\bar x,\bar z)$ which are retractive from no relations
to ON-$3$-frames $\bar a$ with $\bar r$ in $R(\bar a)$
consisting of commuting self-adjoints.
Given a term  $p(\bar x)$ in the language of rings,
let $\tilde{p}(\bar x,\bar z)$ the associated lattice term, that is,
for any frame $\bar a$
\[ \omega_{\bar a}p(f_1,\ldots, f_n)
=\tilde{p}(\omega_{\bar a}f_1,\ldots,\omega_{\bar a} f_n) 
\mbox{ for } f_i \in \End(a_1).\]
Observe that, due to uniqueness of  occurrences of variables $\bar x$
(Subsections~\ref{coo} and  \ref{onframe} and Lemma~\ref{adj}), 
$\tilde{p}(\bar x,\bar z)$ is constructed from $p(\bar x)$
in p-time.
Let the ortholattice term $\hat{p}(\bar x,\bar z)$
be obtained from $\tilde{p}(\bar x,\bar z)$
substituting first $s_i(\bar x,\bar z)$  for $x_i$,  then
$t_j(\bar z)$ for $z_j$ and  $t_{1j}(\bar z)$ for $z_{1j}$.
Let $p^\#(\bar x,\bar z)=0$ be the ortholattice 
identity equivalent to the conjunction of
$\hat{p}(\bar x,\bar z)\cap z_1=0$ and $\hat{p}(\bar x,\bar z)+z_1=z_1+z_2$
(see Fact~\ref{ol}).
Observe that this is still obtained in p-time from $p(\bar x)$.
Then  the following are equivalent 
\begin{enumerate}
\item $p^\#(\bar x,\bar z)=0$ is valid in $\lh$.
\item For all $H\in \hr$,  ON-$3$-frames $\bar a$ of $\la(H)$
and commuting self-adjoint $r_1, \ldots ,r_n \in R(\bar a)$
one has $\tilde{p}(r_1,\ldots ,r_n)$ invertible in  $R(\bar a)$. 
\item For all $H\in \hr$  and commuting self-adjoint
endomorphisms $f_1,\ldots ,f_n$ of $H$
one has $p(f_1,\ldots ,f_n)$ invertible.
\item $p(\rho_1, \ldots ,\rho_n)\neq 0$ for all $\rho_1,\ldots ,\rho_n\in
\mathbb{R}$.
\end{enumerate}  
The equivalence of (i), (ii), and (iii)
is obvious by the above, that of (iii) and (iv) by
the Spectral Theorem: the $f_i$ have a common basis of eigenvectors. 
Thus, we have obtained $(*)$.

\subsection{Reducing {\rm SAT}$_{\lt(H)}$ to {\rm FEAS}$_{\mathbb{Z},\mathbb{R}}$}\lab{feas1}
We continue the proof in Subsection~\ref{SF}
to include orthocomplementation
in the reduction from {\rm SAT}$_{\lt(H)}$ to {\rm FEAS}$_{\mathbb{Z},\mathbb{R}}$ 
 for the  case of fixed dimension
of $V_{\mathbb{R}}=H$. Here, we may assume $H=\mathbb{R}^d$
with the canonical inner product.
Observe that  $\Sp(B)^\perp = \Sp(C)$ 
if and only if $C^tB=0$ and $BY+CZ=I$
for suitable $Y,Z$. Here $C^t$ denotes the transpose of $C$.
Thus, we translate
  $y=z^\perp$
 into 
\[\hat{z}^t\hat{y} =0  \,\wedge\; \exists Y \exists Z.\, \hat{y}Y+\hat{z}Z=I.\]  
In view of Fact~\ref{bsat} 
we have 
a polynomial $p_\tau(x,y)$ 
over $\mathbb{Z}$ 
and for each $d$ the
reduction $\tau_d$ in time $p_\tau(x,d)$
from SAT$_{\lt(H)}$ to {\fs}, where $H \in \hr$ and $\dim H=d$.

\subsection{Reducing {\rm uREF}$_{\lt(H)}$ to {\rm SAT}$_{\lt(H)}$}\lab{feas2}
Given $d$, let $\phi_d(\bar z)$ be the conjunction
of equations defining $d$-frames $\bar a$ such that 
$a_\bot=0$ and $a_\top=1$. Recall the term 
$\delta_d(x,\bar z)$ from Lemma~\ref{discr}.
Now, given an unnested term $T(\bar x)$ in the language of ortholattices,
define
$\sigma_d(T)(\bar x)$ as
\[ \exists \bar z \exists y.\, T(\bar x)=y 
\,\wedge\, \delta_d(y,\bar z)=1 \,\wedge\, \phi_d(\bar z)  
\]
the prenex equivalent of which  is an existentially
quantified conjunction of equations. Then  the following holds.
\begin{itemize}
\item[] For any  $H\in\hr$, $\dim H=d$,
one has $\vb \models \exists \bar x.\,T(\bar x)\neq 0$
if and only if $\vb \models \exists \bar x \,\sigma_d(T)(\bar x)$.
\end{itemize}
Hence, there is
 a polynomial $p_\sigma(x,y)$ over $\mathbb{Z}$
such that 
  $\sigma_d$ is a $p_\sigma(x,d)$-time reduction
from uREF$_{\lt(H)}$ to SAT$_{\lt(H)}$, for each $d$ and
$H \in \hr$, $\dim H=d$.

\subsection{Reducing {\rm uREF}$_\mc{L}$  
to {\fs} and proof of Theorem~\ref{thm}}\lab{feas}
Combining Subsections~\ref{feas1} and \ref{feas2}, the reduction $\rho_d$ 
from {\rm uREF}$_{\lt(H)}$ to {\fs} 
given by
  $\rho_d(T)=\tau_d(\sigma_d(T))$
is carried out in time $p_\tau(p_\sigma(x,d),d)$.
And, in view of  Lemma~\ref{ldim},  a p-time reduction
 of {\rm uREF}$_{\lh}$ to {\fs}  is obtained
applying $\rho_d$ to $T$ where
$d=o(T)$. Since $o(T) \leq |T|$, obviously,
a polynomial bound in terms of $|T|$
is given by    $p_\tau(p_\sigma(x,x),x)$.

By Fact~\ref{ured} we derive a p-time reduction from 
{\rm REF}$_\mc{L}$  
to {\fs} via {\rm uREF}$_\mc{L}$.
The converse reduction is Subsection~\ref{thm2}.

\subsection{Equivalences for arbitrary $\vb$}
\begin{corollary}\lab{dsat}
Any two of the following  decision problems are p-time equivalent: {\fs},
${\rm REF}_{\lt(H)}$, ${\rm SAT}_{\lt(H)}$ where $H \in \mc{H}$, $\dim H\geq 3$,
arbitrary.
\end{corollary}
Using 
 non-deterministic BSS-machines, this has been derived in
 Corollaries 2.8 and  2.12 in \cite{jacm}. 
\begin{proof}
The reduction from 
${\rm REF}_\vb$ via  ${\rm SAT}_\vb$ to {\fs} 
is given by Subsections~\ref{feas2} and \ref{feas1}
(in the complex case one has to use real and imaginary
parts for the encoding into $\mathbb{R}$).
Conversely,  
for fixed $H \in\mc{H}$, the equivalence of (iv) and
of the statements in  (ii) and  (iii) 
in Subsection~\ref{thm2} is valid (since ON-$3$-frames in lower
sections of $\vb$ exist).
\end{proof}

\subsection{$2$-distributive modular ortholattices}
A (ortho)lattice $L$ is $2$-\emph{distributive} if 
the identity \[x \cap \sum_{i=0}^2 x_i
= \sum_{j\neq k} x\cap (x_j+x_k)\]
holds in $L$. Examples of such are, for any cardinal $n>0$, 
the MOLs  MO$_n$ of height $2$ with $n$ pairs $a,a^\perp$ of atoms.
Put $M_0={\bf 2}$.
\begin{fact}\lab{2d} \begin{enumerate}
\item Any finite MOL  is $2$-distributive.
\item A $2$-distributive MOL  is subdirectly irreducible
if and only if it is isomorphic to some {\rm MO}$_n$.
\item For any class $\mc{C}$ of $2$-distributive MOLs,
${\sf V}(\mc{C})$ is generated by some {\rm MO}$_n$,
$n \leq \omega$.
\end{enumerate}
\end{fact}
\begin{proof}
Recall that  in MOLs congruences
are the same as lattice congruences.
Thus, if $L$ is a finite MOL,
it is isomorphic to a direct product of
subspace lattices $L_i$ of finite irreducible projective spaces,
and the $L_i$ are  MOLs. Thus,
according to Baer \cite{baer} the $L_i$ are of height $\leq 2$.
Now let $L$ be a
$2$-distributive  subdirectly irreducible MOL.
According to J\'{o}nsson \cite{jonsson} 
any complemented modular lattice $L$ embeds into a direct product
of subspace lattices $L_i$ of irreducible projective spaces,
$L_i \in {\sf V}(L)$. Being a subdirectly irreducible lattice,
 $L$ embeds into some $L_i$
which is of height 
$\leq 2$ by $2$-distributivity. Thus, $L$ is of height $\leq 2$
whence isomorphic to some {\rm MO}$_n$.  (iii) follows from the fact
that MO$_n$ embeds into MO$_m$ for $n \leq m$ 
and that any variety is generated by its at most
countable subdirectly irreducibles. \end{proof}

\begin{proposition}
If $\mc{C}$ consists of $2$-distributive MOLs
then {\rm SAT}$_\mc{C}$ and {\rm REF}$_\mc{C}$ are $\mc{NP}$-complete.
\end{proposition}
\begin{proof}
That both are in $\mc{NP}$ is \cite[Proposition 1.19]{jacm}.
$\mc{NP}$-hardness of SAT$_\mc{C}$ is \cite[Proposition 1.16]{jacm}.
Now, ${\sf Eq}(\mc{C}) ={\sf Eq}({\rm MO}_n)$ 
for some $n$ by (iii) of Fact \ref{2d} whence 
its decision problem is co$\mc{NP}$-hard by 
\cite[Theorem 1.20]{jacm}.
\end{proof} 

\section{Preliminaries: Part III}\lab{ring}
\subsection{Translations}\lab{trans}
Consider the quantifier free parts  $\Lambda_1$ and $\Lambda_2$
of two first order languages with equality (also denoted by $=$).
Let $\bar z$ be  a string of specific variables in $\Lambda_2$.
Suppose that for each variable $x$ in $\Lambda_1$ 
there is given a term $\tau(x)(\bar z)$ in $\Lambda_2$
and, for 
each
operation symbol $f$ in $\Lambda_1$ and term $f(\bar x)$,
 a term $\tau(f(\bar x))(\tau(\bar x),\bar z)$
where $\tau(\bar x)$ denotes the string of $\tau(x_i)$'s.
Then there is a unique extension to a map (also denoted
by $\tau$) from $\Lambda_1$ to $\Lambda_2$ such that
for any $n$-ary operation symbol $f$ and terms 
$t_i(\bar y_i)$ in $\Lambda_1$  
\[\begin{array}{l}\tau(f(t_1(\bar y_1), \ldots, f(t_n(\bar y_n)))=\\ 
{}\hfill =\tau(f)( \tau(t_i(\bar y_1))(\tau(\bar y_1),\bar z), \ldots, \tau(t_n(\bar y_n))(\tau(\bar y_n),\bar z)  ,\bar z) \end{array} \]
(equality of terms)
and such that an equation
$t(\bar x)= s(\bar x)$
is translated into the equation  $\tau(t(\bar x)) =\tau(s(\bar x))$, and,
 finally, 
such that  $\tau$ is compatible
with the propositional junctors.
\begin{fact}
The translation $\tau(T(\bar x))$ of unnested terms $T(\bar x)$
is carried out in p-time.
\end{fact}
Recall the translation $\theta_i$ (Fact~\ref{ured})
within $\Lambda_i$  of
terms  into unnested terms. 
Structural induction yields the following.

\begin{fact}
For any term $t(\bar x)$ in $\Lambda_1$, 
 $\tau(t(\bar x))=y$ is  logically equivalent
to  $\tau(\theta_1(t(\bar x)))=y$. 
\end{fact}
In view of Fact~\ref{ured},
as an immediate consequence one obtains the following for
classes  $\mc{A}_i$ of algebraic structures in the signature
of $\Lambda_i$.

\begin{fact}\lab{trans1}
If $\tau$ restricts to a reduction of {\rm REF}$_{\mc{A}_1}$
to {\rm REF}$_{\mc{A}_2}$ then  $\theta_2 \circ \tau$ 
yields a p-time reduction of {\rm uREF}$_{\mc{A}_1}$
to {\rm uREF}$_{\mc{A}_2}$. 
\end{fact}
\subsection{$*$-regular rings}\lab{ring1}

A $\ast$-\emph{ring} is a ring (with unit) having as additional
operation an
involution $a \mapsto a^*$. This involution is \emph{proper}
if  $aa^*=0$ only for $a=0$.
A ring  $R$ with proper involution is  $\ast$-\emph{regular}
if for any $a\in R$, there is  $x\in R$ such that $axa=a$;
equivalently, for any $a \in R$ 
there is a [Moore-Penrose]
\emph{pseudo-inverse}
(or Rickart relative inverse) $a^+ \in R$, that is 
\[a=aa^+a,\quad a^+=a^+aa^+,\quad(aa^+)^\ast=aa^+,\quad (a^+a)^\ast=a^+a\]
cf. \cite[Lemma 4]{Kap}.
In this case, $a^+$ is uniquely determined by $a$ 
and will be considered an additional unary
fundamental operation $q(a)=a^+$ of the $\ast$-regular ring $R$.
Thus, $*$-regular rings form a variety.
An element $e$ of a $\ast$-regular ring is a
\emph{projection} if $e=e^2=e^\ast$.
For such $e$, one has $e=e^+$; 
also, each $aa^+$ is a projection.

\begin{lem}~\lab{nest1}
Within the class of $*$-regular
rings with pseudo-inversion, any conjunction  of equations
is equivalent to a single one of the form $t=0$, to be obtained in p-time.
\end{lem}
\begin{proof}
Let $R$ be any $*$-regular ring. If $e^2=e$ and $1=er$ for some
$r$ then $e=e^2r=er=1$.  Thus, by induction, if $\prod_{i=1}^n e_i=1$ with idempotents $e_i$ then $e_i=1$ for all $e_i$.
Now,
the given equations may be assumed of the form $t_i(\bar x)=0$.
Put $t(\bar x)=1-\prod_i (1- t_it^+_i)$.
Then $t(\bar a)=0$ if and only if $1-t_it_i^+(\bar a)=1$ 
for all $i$, that is $t_it_i^+(\bar a)=0$ which means  $t_i(a)=0$.
\end{proof}

The endomorphisms of a finite dimensional 
Hilbert space $H$ over $\mathbb{F}\in\{\mathbb{R},\mathbb{C}\}$  form a 
\qr  
$\End(H)$ where $f^*$ is the adjoint of $f$ and
where the projections
are the orthogonal projections $\pi_U$  onto subspaces $U$.
Moreover, 
 $f^+$ 
is given by $f^+|_{W^\perp}=0$ 
and $f^+|_W:W \ra U$ being 
 the inverse of $f|_{U}\colon U\to W$
  where $U=(\ker f)^\perp$  and
$W= \im f$. $\End(H)^+$ will denote
$\End(H)$ endowed with this additional operation.

\subsection{Ortholattices of projections}\lab{ring3}

 The projections of a $\ast$-regular
 ring $R$  form a
modular ortholattice $\La(R)$ where the partial order
is given by $e \leq f \Leftrightarrow fe=e \Leftrightarrow ef=e$,
least and greatest elements as $0$ and $1$, 
 join  and meet   as
\[ e \cup f=
 f+ (e(1-f))^+ e (1-f),\quad e \cap f = (e^\perp\cup f^\perp)^\perp    \] 
and the  orthocomplement as  $e^\perp=1-e$.
Moreover, for a finite dimensional Hilbert space $H$,
  $e \mapsto {\sf im}\,e$ is an isomorphism  of $\La(\End(H))$
onto $\La(H)$. Thus, associating with each  ortholattice variable $x$
a ring variable $\hat{x}$ and replacing each occurrence of $x$ by
 $\hat{x}\hat{x}^+$ one obtains an (EXP-time) 
interpretation of $\La(R)$ within $R$, uniformly for
all $*$-regular rings.

\subsection{Capturing pseudo-inverse}\lab{ring4}
Recall from Subsections~\ref{onframe} and \ref{adjoints} 
 the concept of an ON-$3$-frame $\bar a$
in $\La(H)$ and its associated 
coordinate ring $R(\bar a)$ with involution $r \mapsto r^{\dagger_{\bar a}}$,
isomorphic to $\End(a_1)$ as a $\ast$-ring via $\omega_{\bar a}$.
We have to capture  the additional operation
$f \mapsto f^+$ of pseudo-inversion on $\End^+(H_1)$,
$H_1$ a subspace of $H$ such that $\dim H \geq 3 \dim H_1$.
Indeed, for
any ON-$3$-frame $\bar a$ of $\La(H)$  (and there is one such that $a_1=H_1$)
 and for any  $r=\omega_{\bar a} f$, $f\in \End(a_1)$, 
one has $ {\sf ker }f={\sf ker}(r,\bar a)$
and  ${\im }f  = {\sf im}(r,\bar a)$ where 
 \[ {\sf ker}(x,\bar z):= 
x \cap z_1 \mbox{ and } 
{\sf im}(x,\bar z)  := 
((x+z_1)\cap z_2 +z_{12})\cap z_1.\]
Indeed, $\omega_{\bar a}(f)\cap a_1=\{v -\vep_{\bar a}f(v)\mid v \in a_1,\, \vep_{\bar a}f(v)=0\}=\{v\in a_1\mid f(v)=0\}$
and $(\omega_{\bar a}(f)+a_1)\cap a_2= \{ \vep_{\bar a}f(v)\mid v \in a_1\}$ 
 whence one has $\im(\omega_{\bar a}(f),\bar a)
= \{x - \vep_{\bar a}(x) +\vep_{\bar a}f(v)\mid x,v \in a_1\}\cap a_1
 = \{x \in a_1\mid \exists v \in a_1.\,\vep_{\bar a}(x) =\vep_{\bar a}f(v)\}
=\im f$.

Thus, with $s(x,\bar z)$ from Fact~\ref{inverse} 
one obtains  an  ortholattice term capturing pseudo-inversion, uniformly for all $H_1 \in \mc{H}$.  Define
 \[ \begin{array}{lcl}
\im^+(x,\bar z) &:=&(z_1\cap \ker(x,\bar z)^\perp+z_{12})\cap (z_1+z_2)\\   
 \psi_0(x,\bar z)& :=&s(x,\bar z) \cap (\im(x,\bar z) + \im^+(x,\bar z))\\  
\psi(x,\bar z)&:=& \psi_0(x,\bar z) +z_1\cap ({\sf im}(x,\bar z))^\perp.
\end{array}\]
\begin{lemma}
For
each  ON-$3$-frame $\bar a$ in $\lt(H)$, $H \in \mc{H}$, 
and all  $f \in \End(a_1)$ one has 
 $\omega_{\bar a}(f^+)=\psi(\omega_{\bar a}(f),\bar a)$.
\end{lemma}
\begin{proof}
Put $r=\omega_{\bar a}(f)$.
Then $\im^+_2(r,\bar a)= \vep^{ \bar a}_{12}(\im f^+) $ and,
in view of Fact~\ref{inverse}, 
$\psi_0(r,\bar a)$ consists of the
$f\vep_{\bar a}^{-1}(w) -w$ where 
$ w \in a_2$ and $w=\vep_{\bar a}f^+(v_2)$ 
for some $v_2\in \im f$ and 
$ f \vep_{\bar a}^{-1}(w)= f(v_1)$ for some $v_1 \in \im f^+$. 
It follows that $f(v_1)= f  f^+(v_2)=v_2$
and $f \vep_{\bar a}^{-1}(w)-w= v_2- (\vep_{\bar a} f^+)(v_2)$,
 that is, 
$\psi_0(r,\bar a)= \{v_2 -\vep_{\bar a} f^+(v_2)\mid v_2 \in \im f\}$.
\end{proof}

\section{Complexity of the equational theory
of $\mc{R}$}
For  finite dimensional real or complex Hilbert spaces $H$,
let  $\End^+(H)$ denote the  endomorphism $*$-ring with pseudo-inversion;
let $\mc{R}$ denote the class of all these.
In analogy to Fact~\ref{neu}  we have the following.

\begin{fact}\lab{neu2}
For any  $\mc{C} \subseteq 
 \mc{R}$,  
the variety ${\sf V}(\mc{C})$
either is generated by one or two members of  $\mc{C}$
or it equals $\mc{R}$
and is generated by any family in 
$\mc{R}$
  having
unbounded dimensions.  
\end{fact}
We mention that 
${\sf V}(\mc{R})$ contains all 
finite Rickart $C^*$-algebras  \cite[Theorem 2]{hs}.
According to \cite[Theorem 22]{SAT}, SAT$_{\mr}$ is undecidable.
Our main result is as follows.
\begin{theorem}\lab{eqring} The decision problems
{\rm REF}$_{\lh}$, {\rm uREF}$_{\lh}$, 
  {\rm REF}$_{\mr}$,  {\rm uREF}$_{\mr}$,  and 
$\rm{FEAS}_{\mathbb{Z},\mathbb{R}}$
are pairwise  p-time equivalent;
in particular, the equational theory of   ${\mr}$
is decidable.
 \end{theorem}
Decidability of the equational theory is also shown in unpublished joint work with
Marina Semenova by reduction to 
decidability of the reals.
\subsection{Outline}\lab{out2}
The reduction from
$\rm{FEAS}_{\mathbb{Z},\mathbb{R}}$ to  {\rm REF}$_{\mr}$ 
is established in  Subsection~\ref{u1},
directly. 
Concerning the reduction in the converse direction,
recall that the reduction of REF$_{\lh}$ to $\rm{FEAS}_{\mathbb{Z},\mathbb{R}}$ 
relied on the fact that $t(\bar x)=0$ fails in $\mc{L}$
if it does so in some $\lt(H)$, $\dim H$ polynomially bounded
by the length of $t(\bar x)$. We make use of this,
reducing via REF$_\mc{L}$.  Though, expressing pseudo-inversion within
the coordinate ring of an ON-$3$-frame requires multiple
occurrences of the principal variable causing exponential
blowup when translating terms.
Therefore, we reduce 
{\rm REF}$_{\mc{R}}$ via
{\rm uREF}$_{\mc{R}}$   to   {\rm uREF}$_{\mc{L}}$  in Subsection~\ref{uref}.
The Theorem then follows with 
Subsection~\ref{feas}.

A term for  quasi-inversion is essential
for the above hardness result and similar ones in 
fixed finite dimension.
Considering polynomial identities
for matrix rings in fixed dimension, lower bounds on
proof complexity have been established by Tzameret et al.,
cf.  \cite{tza}.

 \subsection{Reduction from $\rm{FEAS}_{\mathbb{Z},\mathbb{R}}$
to ${\rm REF}_\mr$}\lab{u1}
Recall from Subsection~\ref{ring1} that, for any $f \in \End(H)$,
 $ff^+$ is a projection
such that $\im f=\im ff^+$.
Thus, according to Subsection~\ref{ring3},
there is a binary term $x\cap y$ in the language of $*$-rings
with pseudo-inversion
such that
$f\cap g$ is the orthogonal projection onto 
$\im f \cap \im g$,  for all $f,g \in \End(H)$ and $H \in \mc{H}$.
Similarly, the term $k(x):=1- x^{*}x^{*+}$ 
is such that $k(f)$ is the orthogonal projection
onto $\ker f = (\im f^*)^\perp$.

\begin{lemma}\lab{lu1}
For all $\mc{C}\subseteq \mc{R}$, containing a non-zero member,
there is a uniform p-time reduction 
of ${\rm FEAS}_{\mathbb{Z},\mathbb{R}}$
to ${\rm REF}_{\mc{C}}$.
\end{lemma}
\begin{proof}
Given a multivariate polynomial $p(\bar x)$ with integer
coefficients, choose new variables $y_i$ and
define the following term
in the language of $*$-rings with pseudo-inversion
(considering $p(\bar x)$ a ring term):
\[ p^\circ(\bar x,\bar y):= k(p(x_1q +(1-q),\ldots, x_nq +(1-q))  \]
where
\[q:=q(\bar x,\bar y):=
k(p(\bar x)) \cap
\bigcap_i k(x_i+x_i-(y_i+y^*_i)) \cap
\bigcap_{i,j} k(y_iy_j^* -y_j^*y_i) \]  
and where $\bigcap$ stands for suitable iterations of $\cap$.
Observe that for any substitution $\bar x \mapsto \bar f$,
$\bar y \mapsto \bar g$ within $\End(H)$, the $f_iq(\bar f,\bar g)+(1-q(\bar f,\bar g))$
are self-adjoint, equal   to $\frac{1}{2}
(g_iq(\bar f,\bar g)+g^*_iq(\bar f,\bar g)) +
1-q(\bar f,\bar g)$, and commute pairwise. On the other hand, given any self-adjoint and pairwise commuting
$f_i=g_i$ one has $q(\bar f,\bar g)=\id$ and it follows that
 $p^\circ(\bar f,\bar g)$ is the projection onto 
${\sf ker}\,p(\bar f)$ and that   $p^\circ(\bar f,\bar g)=0$
if and only if $p(\bar f)$ is invertible.
Now, in view of the Spectral Theorem as in the proof of Theorem~\ref{thm},
given $H\in \mc{C}$, one has  
$p^\circ(\bar f,\bar g)= 0$ for all $\bar f,\bar g$ in $\End(H)$
if and only if $p(\bar x)$ 
has no zero in   $\mathbb{R}$. 
\end{proof}

\subsection{Reducing {\rm REF}$_{\mc{R}}$ via
{\rm uREF}$_{\mc{R}}$   to 
{\rm uREF}$_{\mc{L}}$  
 and proof of Theorem~\ref{eqring}}\lab{uref}  
In view of Subsections~\ref{coo}, \ref{onframe}, and \ref{ring4},
for each operation symbol
$g$ (with associated basic equation $y=g(\bar x)$)  in the language
of $\mc{R}$, there is an ortholattice term  $\hat{g}(\bar x,\bar z)$
such that for any $H \in \mc{H}$ and ON-$3$-frame
$\bar a$ of $\lt(H)$ 
one has, for all $f_i \in   \End(a_1)$ and 
$r_i=\omega_{\bar a} f_i$ in $R(\bar a)$,
\[\End(a_1) \models g(\bar f)=f_0\;\Leftrightarrow 
 \;R(\bar a)\models \hat{g}(\bar r,\bar a)=r_0.\]
Recall Lemma~\ref{retr} and
 the retractive terms $\bar a(\bar z)$ for ON-$3$-frames $\bar a$ 
and $r_i(x_i,\bar z)$ for  elements  of $R(\bar a)$.
For each operation symbol $g$ in the language of $\mc{R}$ define
\[\tau(g(\bar x)):= \hat{g}(\bar r(\bar x, \bar a(\bar z)),\bar a(\bar z)),  \]
where $\bar r(\bar x,\bar a(\bar z))$ is the string of the $r_i(x_i,\bar a(\bar z))$.  According to Subsection~\ref{trans},
this defines a translation $\tau$ from the language
of $\mc{R}$ to that of $\mc{L}$. 
 Since any $H_1 \in \mc{H}$
occurs as $a_1$  for some ON-$3$-frame  $\bar a$ of $\La(H)$,
for some extension  $H$ in $\mc{H}$,
$\tau$
provides a (EXP-time) reduction of the equational
theory of $\mc{R}$ to that of $\mc{L}$;
that is, of REF$_{\mc{R}}$ to REF$_{\mc{L}}$. 
Thus, by Fact~\ref{ured} and  Fact~\ref{trans1},
$\tau$ provides a p-time reduction from
REF$_{\mc{R}}$ (via
uREF$_{\mc{R}}$) to uREF$_{\mc{L}}$.

Together with Subsection~\ref{feas}
one obtains a p-time reduction from 
REF$_{\mc{R}}$ to {\fs} via uREF$_{\mc{R}}$ and uREF$_{\mc{L}}$. 
The converse reduction is  Lemma~\ref{lu1}. 
The reduction of {\fs} to REF$_{\lh}$ is Subsection~\ref{thm2}.

\subsection{Dimension bounds in {\rm REF}$_{\mc{R}}$ }
With Lemma~\ref{ldim}, Subsection~\ref{uref}
yields the following. 

\begin{corollary}
There is a polynomial  $p(x)$ 
such that an equation $t(\bar x)$
fails in $\mc{R}$ if and only if it does
so in $\End^+(H)$ for some $H$ with $\dim H\leq 
p(|t|)$. 
\end{corollary}

\section{The category of finite dimensional Hilbert spaces}
Finite dimensional Hilbert spaces have a prominent r\^{o}le 
in the approach to Quantum Mechanics in terms of Category Theory,
see  e.g. Abramsky and Coecke \cite{ab} and, for equational aspects,
Selinger \cite{sel}.

Let $\mc{H}$ denote the  additive category 
of  finite dimensional Hilbert spaces,
enriched by the contravariant functor of adjunction, cf. \cite{ab,heun,sel}.
Let $\mc{H}^+$ arise from $\mc{H}$ by endowing each endomorphism $*$-ring
$\End(H)$ also  with the operation of pseudo-inversion.
Both $\mc{H}$ and $\mc{H}^+$ shall be considered 
as two-sorted  partial  algebraic structures: 
one sort for objects, one for morphisms. Also, we require
unary operations $\delta$ and $\rho$ from morphisms
to objects yielding domain and codomain
as well as $\iota$ associating the identity on $H$
 with the object $H$. 
Also, we have the map $\omega$ associating with objects $H_1,H_2$
the zero map from $H_1$ to $H_2$.
Speaking about \emph{subcategories} we require
closure under the additive structure and these operations.

\emph{Terms} are built from morphism variables $x_i$
and expressions $\iota(v_i)$ and $\omega(v_i,v_j)$ with object
variables $v_i,v_j$,
subject to rules which grant evaluation in $\mc{H}$ respectively
$\mc{H}^+$. Namely, with each  term
$t$ we  associate  object
variables for domain and codomain,
denoted by $\delta(t)$ and $\rho(t)$. We require 
$\delta(\iota(v_i))=v_i=\rho(\iota(v_i))$, $\delta(\omega(v_i,v_j))
=v_i$, $\rho(\omega(v_i,v_j))=v_j$; moreover,
$t=t_2 \circ t_1$ is defined if and only if $\delta(t_2)=\rho(t_1)$
and then $\delta(t)=\delta(t_1)$ and $\rho(t)=\rho(t_2)$;
similarly, for the other symbols for operations on
the sort of morphisms. Compare \cite[Section 5]{sel}.

Defining  \emph{unnested terms} $T=(\phi_T,y_T)$
in analogy to Subsection~\ref{u2}, the  conditions
on $\delta$ and $\rho$  are included as conjuncts
into $\phi_T$.
An assignment $\gamma$  \emph{admissible} for
a term $t$  (unnested term $T$) 
assigns a morphism $\gamma(z)$
and an object $\gamma(v)$
 to each morphism  variable $z$ 
and object variable $v$  occurring in $t$ ($T$)
such that $\gamma(z) \in {\sf Hom}(\gamma(\delta(z)),\gamma(\rho(z))$
for each $z$, $\iota(v)\in \End(\gamma(v))$,
and $\gamma(\omega(v_i,v_j)) \in {\sf Hom}(\gamma(v_i),\gamma(v_j)$
for all $v,v_i,v_j$
 (and such that $\phi_T$ is satisfied).
Such $\gamma$  provides a unique  evaluation $\gamma(t)$
($\gamma(T)$); in particular, $\gamma(x)=\gamma(x_1)\circ \gamma(x_2)$ 
if  $x=x_1 \circ x_2$ occurs in $\phi_T$. 

\emph{Satisfiability} and \emph{refutability} of (unnested) equations
are defined w.r.t. admissible assignments in analogy to
Subsection~\ref{u2}. 
Observe that within $\mc{H}$ ($\mc{H}^+$) any 
equation is equivalent to one where one side is a zero.
This yields the decision problems {\rm SAT}$_{\mc{C}}$,
{\rm REF}$_{\mc{C}}$
and their unnested variants
{\rm uSAT}$_{\mc{C}}$ and 
{\rm uREF}$_{\mc{C}}$
 for subcategories $\mc{C}$ of $\mc{H}$ and $\mc{H}^+$.
Observe the p-time reductions of 
{\rm SAT}$_{\mc{C}}$ to {\rm uSAT}$_{\mc{C}}$
and {\rm REF}$_{\mc{C}}$ to {\rm uREF}$_{\mc{C}}$.
\begin{fact}\lab{obj}
With any equation $\eta$ in the language 
of $*$-rings  one associates in p-time
an equation $\eta'$ in the language of $\mc{H}$ 
such that, for any subcategory $\mc{C}$ of 
 $\mc{H}$  one has $\eta'$
satisfiable respectively refutable in $\mc{C}$
if and only if $\eta$ is so in the  class of
$\End(H)$, $H$ ranging over objects of $\mc{C}$;
similarly, with the language of $\mc{R}$, $\mc{H}^+$, and $\End^+(H)$.
Also, the analogues hold for unnested equations.
\end{fact}
\begin{proof}
Choose a single object variable $v$
 and put $\delta(t)=\rho(t)=v$ for all subterms  $t$ occurring in $\eta$.
\end{proof}

\begin{thm}\lab{hilbcat}
\begin{enumerate}
\item Let $\mc{C}$ be a
subcategory of 
$\mc{H}^+$ respectively $\mc{H}$  with finitely many objects
 including some  non-zero $H$.
In case of $\mc{H}^+$, the decision problems 
 {\rm SAT}$_{\mc{C}}$,  {\rm uSAT}$_{\mc{C}}$,
{\rm REF}$_{\mc{C}}$, and {\rm uREF}$_{\mc{C}}$   are each p-time equivalent to 
$\rm{FEAS}_{\mathbb{Z},\mathbb{R}}$; in case of $\mc{H}$ so are
 {\rm SAT}$_{\mc{C}}$ and {\rm uSAT}$_{\mc{C}}$. In particular,
decidability holds in all cases. 
\item 
 {\rm SAT}$_{\mc{C}}$ and {\rm uSAT}$_{\mc{C}}$ are  undecidable
for $\mc{C}=\mc{H},\,\mc{H}^+$.
\item {\rm REF}$_{\mc{H}^+}$ and {\rm uREF}$_{\mc{H}^+}$ 
are both  decidable and each   p-time equivalent to 
$\rm{FEAS}_{\mathbb{Z},\mathbb{R}}$.
\end{enumerate}
\end{thm}

\begin{proof}
Lemma~\ref{lu1} and Fact~\ref{obj} provide,
for any subcategory $\mc{C}$ of $\mc{H}^+$, 
a p-time reduction of {\fs} to 
{\rm REF}$_{\mc{C}}$.

In (i),  given  $\mc{C}$ with finitely many objects,
let $d$ be the maximal dimension of an object in $\mc{C}$.
Introducing coordinates one obtains a p-time reduction
of {\rm uSAT}$_{\mc{C}}$ to $\rm{FEAS}_{\mathbb{Z},\mathbb{R}}$.
A  p-time reduction of {\rm REF}$_{\mc{C}}$
to {\rm SAT}$_{\mc{C}}$
(and, similarly, for the unnested variants) can be based on the observation that
for $f \in {\sf Hom}(H',H)$ one has $f \neq \omega(H',H)$
if and only if there are $g_i \in \End(H)$
such that $\id_H= \sum_{i=1}^d g_i\circ f$ where $\dim H\leq d$.

Now, let $H \in \mc{C}$ with $d=\dim H$.
${\rm FEAS}_{\mathbb{Z},\mathbb{R}}$ reduces to
{\rm SAT}$_{\mc{C}}$ requiring, in a 
conjunction of equations with unique object variable $v$,
 a system of $d\times d$ $*$-matrix
units $e_{ij}$ such that $\iota(v)$
evaluates to $\sum_{ii} e_{ii}$ and interpreting $\mathbb{R}$
as the set of $r \in \End(H)$ such that $r=r^*=e_{11}re_{11}$
cf.  \cite[Theorem 4.4]{jacm}.

Undecidability in (ii) follows from
Fact~\ref{obj} and \cite[Theorem 22]{SAT}.

In order to prove in (iii) the  reduction to {\fs}, we relate 
subcategories $\mc{C}$ with  $n$ objects  to
$(\End^+(H);\bar \pi)$ where $H$ is a finite
dimensional Hilbert space and $\bar \pi$ an $n$-tuple of
orthogonal projections in $\End(H)$.
Given $H$ and $\bar \pi$,
let $\mc{C}_{H\bar \pi}$ have objects  $H_i=\im \pi_i$; 
then ${\sf Hom}(H_i,H_j)$ consists of the $\pi'_j\circ f \circ\vep_i$,
$f \in \End(H)$, where $\pi'_j\in {\sf Hom}(H,H_j)$ is the orthogonal projection
onto $H_j$ and
 $\vep_i$  the identical
embedding of $H_i$ into $H$.
Given $\mc{C}$ with objects $H_i$
let $H_{\mc{C}}$
be given by $H=\bigoplus^\perp_i H_i$ and $\pi_i$ the orthogonal projection
onto the summand corresponding to  $H_i$. Observe that
$\mc{C}$ is isomorphic to $\mc{C}_{H_{\mc{C}}}$.

To translate an unnested  equation $\eta$ in the
language of $\mc{H}^+$  into an unnested equation $\tau(\eta)$
in the language of $\mc{R}$, consider both morphism
and object 
variables as  variables for $\mc{R}$.
Delete the side conditions on domain and range
and read $\circ, +,-,^*,^+$ as operation symbols
for $\mc{R}$. Replace any $\omega(v_1,v_2)$ by $0$ and
$\iota(v)$ by $vv^+$.
 Replace any morphism variable $z$ in 
the resulting formula by
the $\mc{R}$-term $\hat{z}$ given as
\[ \rho(z)\circ\rho(z)^+ \circ z \circ\delta(z)\circ\delta(z)^+.\]
Now, assume that $\eta$ 
fails in $\mc{H}^+$; then it
does so in some subcategory $\mc{C}$ with  finitely many objects
and $\tau(\eta)$ fails in $H_{\mc{C}}$: Given
a failing assignment obtain one in $H_{\mc{C}}$, namely associate with 
 $z \mapsto f\in{\sf Hom}(H_i,H_j)$, $v \mapsto H_k$ in $\mc{C}$
the assignment
 $z \mapsto \vep_j\circ f \circ \pi'_i$,
 $v \mapsto \pi_k$.  Conversely, assume a failing assignment for
$\tau(\eta)$  in $\End^+(H)$ with values $\pi_i$ for the
terms $v_iv_i^+$, $v_i$ an object variable occurring in $\eta$.
Form $\mc{C}=\mc{C}_{H\bar\pi}$ to obtain
a failing assignment for $\eta$
where $v_k \mapsto H_k=\im \pi_k$ and $z \mapsto 
\pi_j \circ f \circ \vep_i$ if
$\hat{z}$ is evaluated to 
$\pi'_j \circ f \circ \pi_i$.
This provides a p-time reduction of {\rm uREF}$_{\mc{H}^+}$
via  {\rm uREF}$_{\mc{R}}$ to {\fs} according to Subsection~\ref{uref}.
\end{proof}
Omitting pseudo-inversion, any term is equivalent to one 
where $^*$ occurs only in the form $x_i^*$.
Thus,  a polynomial  dimension bound for refutation
can be established, directly.
Namely, for any such term $t$, $\bar f$ in $\End(H)$, and $a \in H$
one has $t(\bar f)(a)= t(\bar g)(a)$ where $g_i=\pi_U \circ f_i\circ 
\vep_U$ and $U$ is the subspace spanned by the  
$s(\bar f)(a)$, $s$ a subterm of $t$.
 This gives the upper complexity bound.

Similarly, consider
 the category $\mc{H}$ and  such term $t(\bar x)$,
admissible substitution $\bar f$ of morphisms, and
$a$ such that $t(\bar f)(a)$ is defined. By recursion,  one defines
a subspace $H_{t\bar f a}$  of $H$ for each $H\in \mc{H}$
to obtain the objects of  a  category $\mc{H}_{t\bar f a}$
in which $t(\bar f)(a)$ evaluates the same  as in $\mc{H}$;
here, one has a polynomial bound on the sum of the dimensions
of objects.  Again, this gives the complexity of {\fs}  as
an upper bound for deciding identities.

\end{document}